\documentclass[11pt,reqno]{amsart}
\usepackage{amssymb,latexsym,amsmath,amsthm,mathrsfs,amsfonts,microtype}
\usepackage{fullpage}

\usepackage{float}
\usepackage{graphicx}
\usepackage{tikz}

\usetikzlibrary{shapes.geometric, angles, quotes}
\usetikzlibrary{calc}

\usepackage{xcolor}

\usepackage{mathtools}
\usepackage{hyperref}

\newtheorem{theorem}{Theorem}
\newtheorem{proposition}{Proposition}
\newtheorem{lemma}{Lemma}
\newtheorem{corollary}{Corollary}

\theoremstyle{definition}
\newtheorem{definition}{Definition}

\newtheorem{remark}[definition]{Remark}

\newtheorem*{remark*}{Remark}

\newtheorem*{example*}{Example}

\newtheorem*{terminology}{Terminology}

\newtheorem*{proof idea}{Idea of the proof}

\theoremstyle{remark}

\newtheorem*{claim*}{Claim}

\begin{document}
	
	\title{On uniform recurrence for hyperbolic automorphisms of the $2$-dimensional torus} 
	
	\author{Georgios Lamprinakis}
	\address{Lund University, Centre for Mathematical Sciences, \newline
	\hspace*{\parindent}Box 118, 221 00 Lund, Sweden}
	\email{georgios.lamprinakis@math.lth.se}

	\author{Tomas Persson}
	\address{Lund University, Centre for Mathematical Sciences, \newline
	\hspace*{\parindent}Box 118, 221 00 Lund, Sweden}
	\email{tomas.persson@math.lth.se}
	
	\author{Alejandro Rodriguez Sponheimer}
	\address{Lund University, Centre for Mathematical Sciences, \newline
	\hspace*{\parindent}Box 118, 221 00 Lund, Sweden}
	\email{alejandro.rodriguez\_sponheimer@math.lth.se}

	\date{}

	\begin{abstract}
We are interested in studying sets of the form
\[
\mathcal{U}(\alpha) \ 
:= \ 
 \left\{ x\in X: \ \exists M=M(x) \geq 1 \text{ such that } \forall N\geq M, \ \exists n\leq N \text{ such that } d(T^nx, x) \leq |\lambda|^{-\alpha N} \right\}
\]
where $(X,T,d)$ is our metric dynamical system and $|\lambda|>1$.
Although a lot of results exist for the one dimensional case, not as
many are known for systems in higher dimensions and especially in the hyperbolic case. We consider $X=\mathbb{T}^2$,  $T(x) = Ax \pmod{1}$, where $A$ is a hyperbolic, area preserving, $2\times 2$ matrix with integer entries and $\lambda$ is the eigenvalue of $A$ of modulus larger than $1$ and we explicitly calculate the Hausdorff dimension of this set.
\end{abstract} 

\maketitle

\section{Introduction}
One of the most important results in dynamical systems is Poincar\'e's recurrence theorem proved by Carath\'eodory~\cite{Caratheodory}. Later on Boshernitzan~\cite{Boshernitzan} quantified this question by giving information about the speed of the asymptotic recurrence. This initiated a more systematic study of recurrence properties of a system, as well as its approximation properties and the (more general) shrinking target problems, dynamical Borel--Cantelli lemmas, return/hitting time etc.

Apart from the asymptotic recurrence, in view of Dirichlet's theorem, it is of interest to study the uniform approximation and recurrence properties of a dynamical system. The famous Dirichlet's theorem states that for any real number $\xi$ and for all integers $N\geq 1$, there exists an integer $1 \leq n \leq N$, such
\[
\| n\xi \| \leq N^{-1}
\]
where $\| \cdot \|$ denotes the distance to the nearest integer. This can be expressed in the language of dynamical systems by considering $T_{\xi} \colon \mathbb{T} \to \mathbb{T} \colon  x \mapsto x+\xi \pmod{1}$ and observing that $\| n\xi \| = \| T_{\xi}(x) - x \|$. That way Dirichlet's theorem be interpreted as a uniform recurrence problem. Kim and Liao \cite{Kim Liao} studied the inhomogeneous version of Dirichlet's theorem. In particular, they considered (among other things) the set
\[
\mathcal{U}_{\alpha}(\xi) := \left\lbrace y\in \mathbb{T} : \ \exists M=M(y) \text{ such that } \forall N\geq M, \ \exists 1\leq n \leq N \text{ such that } \|n\xi - y\| < N^{-\alpha} \right\rbrace
\]
and calculated the Hausdorff dimension of $\mathcal{U}_{\alpha}(\xi)$ (which depends on the irrationality of $\xi$). In \cite{Bugeaud Liao}, Bugeaud and Liao studied the corresponding set for a $\beta$-transformation and explicitly calculated the Hausdorff dimension of the set 
\[
\mathcal{U}(\alpha,y):= \left\{x\in \mathbb{T}: \ \exists M=M(x) \geq 1 \text{ such that } \forall N\geq M, \ \exists n\leq N \text{ such that } d(T_{\beta}^nx, y) \leq \beta^{-\alpha N} \right\}
\]
showing that it is $\frac{(1-\alpha)^2}{(1+\alpha)^2}$ for $\alpha \in [0,1]$ and zero otherwise. The same set was studied by Kirsebom, Kunde and Persson \cite{Maxim Philipp Tomas 2019}, even for more general maps, such as piecewise expanding maps and some quadratic maps. Specifically for $\beta$-transformations, they showed a jump between its Hausdorff and Packing dimension.  Zheng and Wu \cite{zheng2020Uniform} investigated the respective uniform recurrence set for a $\beta$-transformation
\[
\mathcal{U}(\alpha):= \left\{x\in \mathbb{T}: \ \exists M=M(x) \geq 1 \text{ such that } \forall N\geq M, \ \exists n\leq N \text{ such that } d(T_{\beta}^nx, x) \leq \beta^{-\alpha N} \right\}
\]
and they proved that the Hausdorff dimension is again equal to $\frac{(1-\alpha)^2}{(1+\alpha)^2}$. It is worth noting at this point, that there is (as expected) some dimensional connection between certain recurrence and approximation questions, as it is showcased in the above mentioned results.

While there are some results for higher dimensions, most of them are for the expanding case. For example He and Liao \cite{He Liao} gave a formula for the dimension for the asymptotic recurrence when $A$ is a diagonal matrix, not necessarily integer, and with all diagonal elements of modulus larger than $1$. Seemingly, very little is known for the hyperbolic case, with the only result of this type known to the authors being due to Hu and Persson \cite{Hu Tomas}.
Namely, the dimension of the set 
\[
\mathcal{L}(\alpha) := \left\{ x\in \mathbb{T}^2: \ d_{\mathbb{T}^2}(T^nx, x) \leq |\lambda|^{-\alpha n} \text{, for infinitely many } n \right\}
\] 
was studied, where $T(x) = Ax \pmod{1}$, for all $x\in \mathbb{T}^2$ and $\lambda$ is the eigenvalue of the matrix $A$ of modulus larger than $1$ and it was explicitly calculated for the case of an integer, hyperbolic, area preserving, $2 \times 2$ matrix in the following result.
\begin{theorem} \label{Theorem non-uniform case}
Let $A$ be a hyperbolic $2\times 2$ integer matrix with $\det A = \pm 1$ and let $\lambda\in \mathbb{R}$ be its eigenvalue so that $|\lambda| > 1$.
Then
\[
\dim_H \big(\mathcal{L}(\alpha) \big) = 
\begin{cases}
\frac{2}{\alpha + 1} \ , \quad &0 \leq \alpha \leq 1\\
\hfill \\
\frac{1}{\alpha} \ , \quad &\alpha \geq 1
\end{cases}
\]
\end{theorem} 
Note that a phase transition occurs here, related to the different choice of optimal covers, depending on $\alpha$. This phenomenon, of course, does not appear in the one dimensional case where the geometry is much simpler.

Motivated by the above results we study, in analogy, the uniform recurrence properties of a certain family of hyperbolic automorphisms in $\mathbb{T}^2$.
We consider the set 
\[
\mathcal{U}(\alpha) := \left\{x\in \mathbb{T}^2: \ \exists M=M(x) \geq 1 \text{ such that } \forall N\geq M, \ \exists n\leq N  \text{ such that } d_{\mathbb{T}^2}(T^nx, x) \leq |\lambda|^{-\alpha N} \right\}
\]
where $T$ and $\lambda$ are as before.
Our focus is, again, to study the Hausdorff dimension of that set which yields the following result.
\begin{theorem} \label{Theorem main result - isomorphism}
	Let $A$ be a hyperbolic $2\times 2$ integer matrix with $\det A = \pm 1$ and let $\lambda\in \mathbb{R}$ be its eigenvalue so that $|\lambda| > 1$.
	Then
	\[\dim_H \big(\mathcal{U}(\alpha) \big) = 
	\begin{cases}
		2 \frac{(1-\alpha)^2}{(1+\alpha)^2} \ , \quad &0\leq \ \alpha \ \leq 3-2\sqrt{2}\\[12pt]
		
		\frac{(1-\sqrt{2\alpha})^2}{\alpha} \ , \quad &3-2\sqrt{2} \leq \ \alpha \ \leq 2-\sqrt{3} \\[10pt]
		
		\frac{1-3\alpha}{1-\alpha} \ , \quad  &2-\sqrt{3} \leq \ \alpha \ \leq 1/3 \\[10pt]
		
		0 \ , \quad &\alpha \geq 1/3 
	\end{cases}
	\]
	Furthermore, the cardinality of\/ $\mathcal{U}\left(\frac{1}{3}\right)$ is the continuum.
\end{theorem}
Below is an illustration of the dimension of
$\mathcal{U}(\alpha)$. Observe that, as anticipated from the
preceding result in Theorem~\ref{Theorem non-uniform case}, a
phase transition manifests when the dimension drops to
one. Furthermore, even though it may not be immediately clear
from our approach, the phase transition here is also related to
the different optimal cover depending on $\alpha$. Interestingly,
the dimension of $\mathcal{U}(\alpha)$ as a function of $\alpha$
is differentiable everywhere except where the phase transition
appears, i.e.\ for $\alpha = 3-2\sqrt{2}$; however, the second
derivative does not exist for $\alpha =2-\sqrt{3}$.
\begin{center}
	\includegraphics[width=80mm, height=55mm]{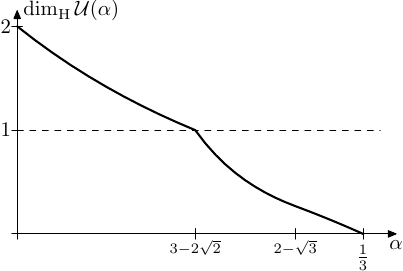}
\end{center}

While the centre of focus and the statement of the results are mostly related to the work in \cite{Hu Tomas},  the proofs are closer in spirit to the preceding pieces of work and mostly with \cite{Bugeaud Liao}. In particular, our strategy revolves around showcasing---as in the $1$-dimensional case---that in an appropriate shift space $\Sigma_{\Gamma}$ (see Section~\ref{Subsection Hyperbolic automorphisms and coding}), the respective set  
 \[
 \mathcal{U}'(\alpha) = \{\underline{x}\in \Sigma_{\Gamma}: \ \exists M=M(\underline{x}) \geq 1 \text{ such that } \forall N\geq M, \ \exists n\leq N \text{ such that } d_{\Sigma}(\sigma^n\underline{x}, \underline{x}) \leq |\lambda|^{-\alpha N}\}
 \]
has the same dimension as the original set. For the maps $\times m \pmod{1}, m \in \{2, 3, 4, \ldots\}$ and $\beta$-transformations, there is a natural connection with the appropriate shift spaces via the $m$-expansion or the $\beta$-expansion of numbers in the interval $[0, 1)$ which essentially allows us to calculate the dimension symbolically. We show that this, even though not as simple as in the one-dimensional case, is also true for the hyperbolic maps under study. This allows us to bypass the more complicated geometry inherited by the higher dimensional setting and more easily handle the complex structure of $\mathcal{U}(\alpha)$.

Although we only consider the $2$-dimensional torus, the methods could
possibly be applied to higher dimensional tori; however, some new
problems arise because of non-conformal stable or unstable directions.

Sections~\ref{Section lower-isomorphism} and \ref{Section upper-isomorphism} are devoted to calculating the Hausdorff dimension of the uniform recurrence set $\mathcal{U}'(\alpha)$. Section~\ref{Section  dim manifold = dim coding space - isomorphism} is devoted to justifying our claim that
 \[
 \dim_H \big(\mathcal{U}(\alpha) \big) = \dim_H \big(\mathcal{U}'(\alpha) \big).
 \]
 
 Without loss of generality we can even assume that $\det A = 1$ and the eigenvalues are both positive. This means that if $0 < \lambda_2 < 1 < \lambda_1$ are the eigenvalues of the hyperbolic matrix $A$, then $\lambda_1 = \lambda >1$ and $\lambda_2 = \lambda^{-1}$. This is sufficient, since the other cases can be treated in the exact same manner as this one and yield, of course, the same dimensional results. Thus, for notational convenience, we are assuming form here on that $\lambda$ is positive and in particular, $\lambda > 1$.

\section{Preliminaries}
\subsection{Shift space} \label{Subsection Shift spaces}
Let $\Sigma_d$, $d\in \mathbb{N}$, $d\geq 2$, denote the full, two-sided shift space corresponding to the alphabet $\mathcal{A} = \{ 0,1, \ldots , d-1 \}$, i.e.\ $\Sigma_d := \mathcal{A}^{\mathbb{Z}}$. We denote the elements of $\Sigma_d$ as 
\[
\underline{x} =(\ldots x_{-n} \ldots x_{-1}x_0x_1 \ldots x_n \ldots) , \quad x_i\in \mathcal{A}, \ i\in\mathbb{Z}
\]
and we call the $x_i$ the $i$-th \textit{letter} or $i$-th \textit{digit} of $\underline{x}$ where $i\in \mathbb{N}$ denotes the position of each digit.  We call \textit{block} or \textit{word} a finite string of letters chosen from the alphabet $\mathcal{A}$ which we denote as 
\[
[a_{m}, \ldots , a_n]
\]
$m,n \in \mathbb{Z}$, $m\leq n$.
If we are mostly interested in the placement of such a block, we just write
\[
[m, \ldots , n]
\]
$m,n \in \mathbb{Z}$, $m\leq n$.
Let $m,n, k \in \mathbb{Z}$, with $m\leq n$ and let $[a_{m}, \ldots , a_n]$ and $[b_{m+k}, \ldots , b_{n+k}]$ be two blocks. We say that the two blocks are equal if 
\[
a_i=b_{i+k} , \quad \text{ for all } i=m,\ldots ,n.
\]
We call the set 
\[C_{[a_{m}, \ldots , a_n]} = \{x\in \Sigma : \ x_i=a_i \text{ for all }  m \leq i\leq n \}
\]
a \textit{cylinder}, where $m,n \in \mathbb{Z}$, $m\leq n$ and $a_i \in \{ 0,1, \ldots , d-1 \}$ for $m \leq i\leq n$.
For simplicity, if it does not cause any confusion, we write $\Sigma$ instead of $\Sigma_d$.

The topology on $\Sigma$ is the product topology and it is a compact, metrizable topological space. 
A compatible metric\footnote{The reason for the choice of this specific metric will become clear in Section~\ref{Section  dim manifold = dim coding space - isomorphism}.} is 
\[
d_{\Sigma}(\underline{x}, \underline{y}) = \lambda^{-k(\underline{x}, \underline{y})}
\]
where $\lambda$ denotes the largest eigenvalue of the matrix $A$ 
and $k(\underline{x}, \underline{y}) = \max \{n\in \mathbb{N}: \ x_{i}=y_{i}, \text{ for all } |i|\leq n  \}$. Let $\sigma$ be the regular shift operator on $\Sigma$, such that $\big(\sigma(\underline{x})\big)_i = x_{i+1}$, for all $i\in \mathbb{Z}$. The shift operator acts continuously on $\Sigma$.  

Let $S$ be a closed subset of $\Sigma$ that is invariant under $\sigma$. Any invariant subset of $\Sigma$ is determined by a countable collection of forbidden words. A subset $S$ is a subshift of finite type if there exists a finite list of forbidden words/blocks $\mathcal{F}$ such that a point $x \in \Sigma$ is in $S$ if and only if $x$ contains no blocks from $\mathcal{F}$. Of course the whole shift space is a subshift of finite type. A forbidden block $w=[w_{m}, \ldots ,  w_{n}]$, $m,n \in \mathbb{Z}$, $m\leq n$, can also be described by a finite collection of larger forbidden blocks, for example,
\[
\{ [w_{m}, \ldots ,  w_{n}0], [w_{m},  \ldots , w_{n}1], \ldots , [w_{m},  \ldots , w_{n}(d-1)] \}.
\]
Thus we can assume if needed, that all the forbidden words are of the same length, equal to that of the longest of the initial forbidden blocks.	
Any subshift of finite type can also be represented by a $d^{\ell-1}\times d^{\ell-1}$ matrix, $\Gamma=(\gamma_{ij})$, with entries in $\{0,1\}$, where $\ell$ is the length of the longest forbidden word and $\gamma_{ij} = 1$ when it corresponds to an allowed block and $\gamma_{ij}=0$ otherwise. The matrix $\Gamma$ is called transition matrix. Therefore, and since subshifts of finite type are playing a major role (see Section~\ref{Subsection Hyperbolic automorphisms and coding}), we denote a subshift of finite type with transition matrix $\Gamma$ as $\Sigma_{\Gamma}$.

Let $\Sigma_{\Gamma}$ be a subshift of finite type with transition matrix $\Gamma$.  We define its topological entropy by
\[
h_{\text{top}}(\Sigma_{\Gamma})= \lim_{n \to \infty} \frac{\log (\# \mathcal{A}_n)}{n}
\]
where $\mathcal{A}_n$ denotes the admissible words of length $n$ for the subshift of finite type $\Sigma_{\Gamma}$ and $\#X$ denotes the cardinality of a set $X$. The next very well known result can be found for example in \cite[Appendix II]{Pesin}.
\begin{proposition} \label{Poposition entropy of SFT}
	Let $\Sigma_{\Gamma}$ be a subshift of finite type with transition matrix $\Gamma$. Let $\lambda$ be the largest, in absolute value, eigenvalue of\/ $\Gamma$. Then,
	\[
	h_{\text{top}}(\Sigma_{\Gamma}) = \log \lambda.
	\]
\end{proposition}

\subsection{Hyperbolic automorphisms and coding} \label{Subsection Hyperbolic automorphisms and coding}
Let $A$ be a hyperbolic, area preserving matrix with integer entries and $T: \mathbb{T}^2 \to \mathbb{T}^2$ be defined by $T(x)= Ax \pmod{1}$, $\forall x\in \mathbb{T}^2$. 
A Markov partition for the system $(\mathbb{T}^2 , T)$ is a finite cover $\mathcal{P} = \{P_0 , \ldots , P_{d-1} \}$ of $\mathbb{T}^2$ such that,
	\begin{enumerate}
		\item each $P_i$ is the closure of its interior, int$P_i$  and $P_i$ is convex
		\item int$P_i \cap \text{int} P_j = \emptyset$
		\item whenever  $x\in \text{int}P_i$ and $T(x)\in \text{int}P_j$, then $W_{P_j}^u \big(T(x)\big) \subset T\big(W_{P_i}^u(x)\big)$
		and $T\big(W_{P_i}^s(x)\big) \subset W_{P_j}^s \big(T(x)\big)$.
	\end{enumerate} 
where, $W_{P_i}^u(y)$ denotes the intersection of the unstable manifold of $y$ with the element $P_i$ and $W_{P_i}^s(y)$ is defined analogously. For such a system, the elements of the partition can be constructed so that they are parallelograms with sides parallel to the eigendirections. 
Let $\mathcal{P} = \{P_0 , \ldots , P_{d-1} \}$ be a Markov partition of the system $(\mathbb{T}^2 , T)$. Then $(\mathbb{T}^2 , T)$ can be represented symbolically by a subshift of finite type, $\Sigma_{\Gamma}$ in $\Sigma_{d}$ corresponding to the transition matrix $\Gamma=(a_{ij})$, $i,j \in \{ 0, 1, \ldots , d-1 \}$ where 
	\[
	a_{ij} = \begin{cases}
		1, \quad &\text{int}P_i \cap T^{-1}( \text{int}P_j) \neq \emptyset\\
		0, &\text{otherwise}
	\end{cases}
	\]
This gives rise to the coding map
\[
\pi \colon \ \Sigma_{\Gamma} \to \mathbb{T}^2  \colon \underline{x} \mapsto \bigcap_{j\in \mathbb{Z}} T^{j}(P_{x_j}) \ , \quad \ \text{ for } \underline{x}=(\ldots x_{-n} \ldots x_{-1}x_0x_1 \ldots x_n \ldots)
\]
so that $\pi \circ \sigma = T \circ \pi$.

Adler \cite{Adler} proved that for any such hyperbolic system, we can find a Markov partition so that the spectral radius of the transition matrix is equal to that of the original matrix $A$. For our purpose the results in \cite{Snavely} are also sufficient. In particular we will use the following result.
\begin{theorem}[Snavely] \label{Theorem Snavely}
Let $A$ be a hyperbolic $2\times 2$ integer matrix acting on $\mathbb{T}^2$ and let $\lambda$ be its largest, in absolute value, eigenvalue. Then there is always a Markov partition for which the transition matrix $\Gamma$ has the same largest, in absolute value, eigenvalue.
\end{theorem}

\begin{corollary} \label{Corollary admissible blocks are almost lambda to the n}
Consider the system $T\colon x\mapsto Ax \pmod{1} \colon \mathbb{T}^2 \to \mathbb{T}^2$. Let $\Sigma_{\Gamma}$ be the  corresponding coding space and $\lambda$ be the eigenvalue of $A$, of modulus larger than $1$. Then, the entropy of\/ $\Sigma_{\Gamma}$ is equal to $\log \lambda$.
\end{corollary}

\section{Lower bound} \label{Section lower-isomorphism}
In this section we are going to prove that
\begin{equation} \tag{I}  \label{Equation I: lower bound}
\dim_{H} \big( \mathcal{U}'(\alpha) \big) \ 
\geq \  
\begin{cases}
2 \frac{(1-\alpha)^2}{(1+\alpha)^2} \ , \quad &0\leq \ \alpha \ \leq 3-2\sqrt{2}\\[12pt]

\frac{(1-\sqrt{2\alpha})^2}{\alpha} \ , \quad &3-2\sqrt{2} \leq \ \alpha \ \leq 2-\sqrt{3} \\[10pt]

\frac{1-3\alpha}{1-\alpha} \ , \quad  &2-\sqrt{3} \leq \ \alpha \ \leq 1/3 \\[10pt]

0 \ , \quad &\alpha \geq 1/3
\end{cases}
\end{equation}

We will investigate the lower bound by choosing a sequence of natural numbers which will, essentially, be denoting the placement and the length (i.e.\ the number of consecutive digits each one of the fixed block occupies) of the fixed block that are created by the recurrence condition of $\mathcal{U}(\alpha)$. This placement, will create, for each case, sufficient freedom, i.e.\ sufficiently many positions (in between the fixed blocks of course) in which we can freely (apart from original restrictions of the subshift of finite type $\Sigma_{\Gamma}$) choose the digits,  to have---at least---positive dimension. We also call these free positions, simply free digits. In fact, by choosing an appropriate sequence, this method will gives us the optimal lower bounds for the dimension. 

\begin{remark}
Zero obviously is a lower bound for the dimension. Its role here is to emphasize the fact that for $\alpha \geq 1/3$ we cannot have a better lower bound.
\end{remark}

\begin{proof}[Proof of~\eqref{Equation I: lower bound}]
Let $\theta >1$. We choose\footnote{Even though in our proof $\theta$ is not necessarily a natural number and thus the $n_k$'s may not  be natural numbers, it is not very hard to see that one gets the same results by using the integer part of each $n_k$. The same is true for the $\alpha n_k$'s. To avoid further strain in the notation we exclude it in our proof.} $n_k = \theta^{k}$, $k\in \mathbb{N}$. 
Define the set
\[
\overline{\mathcal{U}}(\alpha):= \left\lbrace \underline{x}\in \Sigma_{\Gamma} : \ d_{\Sigma}(\sigma^{n_k}(\underline{x}), \underline{x})< \lambda^{-\alpha n_{k+1}}, \ \forall k\in \mathbb{N} \right\rbrace.
\]
Then, we obviously have that
\[
\overline{\mathcal{U}}(\alpha) \subset \mathcal{U}'(\alpha)
\]
which means that it is sufficient to get a lower estimate for the dimension of $\overline{\mathcal{U}}(\alpha)$. The condition that an $\underline{x}$ must satisfy for each $k\in \mathbb{N}$ in order to be an element of $\overline{\mathcal{U}}(\alpha)$ (with the given metric in the shift space) imposes that the  blocks 
\[
[n_{k}- \alpha n_{k+1}, \ldots , n_{k} +\alpha n_{k+1}]
\] 
must be equal to the blocks 
\[
[- \alpha n_{k+1}, \ldots , \alpha n_{k+1}].
\] 
In what follows, we will study the behaviour of these fixed blocks and how different conditions on $\alpha$ give different behaviour and thus dimension.

Firstly, to secure some freedom, we need that
\begin{equation} \label{eq. condition 1 (for all alpha)}
n_k \geq \alpha n_{k+1} \ , \quad \text{for all } k\in \mathbb{N}.
\end{equation} 
Indeed, if $n_k < \alpha n_{k+1}$ then there would be no free digits left since
\[
n_k - \alpha n_{k+1} \to -\infty \ , \  \text{ as } k\to +\infty
\]
and of course, by definition
\[
n_k + \alpha n_{k+1} \to +\infty \ , \ \text{ as } k\to +\infty.
\]
Inequality \eqref{eq. condition 1 (for all alpha)} implies that
\begin{equation} \tag{1$'$} \label{eq. condition 1' (for all alpha)}
1<\theta <1/\alpha.
\end{equation}
Observe that \eqref{eq. condition 1' (for all alpha)} suggest that $\alpha <1$.
Furthermore, this means that the fixed blocks
\[
[n_{k}- \alpha n_{k+1}, \ldots , n_{k} +\alpha n_{k+1}]
\]
live in the right-hand side, i.e.\ in the positions after the position $0$.

%%%%no overlaps%%%
Firstly we assume that 
\begin{equation} \label{eq. case 1}
n_k + \alpha n_{k+1} <  n_{k+1} - \alpha n_{k+2} \ , \quad \text{for all } k\in \mathbb{N}.
\end{equation}
This implies that the fixed blocks do not intersect each other (see Figure~\ref{Figure  no overlap}). 
Now \eqref{eq. case 1} implies that 
\begin{equation} \tag{2$'$} \label{eq. case 1'}
1+\alpha \theta < \theta - \alpha \theta^2.
\end{equation}
The only way that we can find a $\theta$ that satisfies relation \eqref{eq. case 1'} is if $0<\alpha<3-2\sqrt{2}$. Indeed, in order for the inequality \eqref{eq. case 1'} to be true for some $\theta >1$, we need the quadratic polynomial that is created to have real roots. This corresponds to $1- 6\alpha +\alpha^2$ being positive. Thus, since  $\alpha <1$, we must have that $\alpha < 3-2\sqrt{2}$, which is the smallest of the two (positive) roots of the polynomial $1- 6\alpha +\alpha^2$.

Let $\underline{x} \in \overline{\mathcal{U}}(\alpha)$. We define for every $k\geq 2$
\[
\mu \big( C_m(\underline{x}) \big) := 
\begin{cases}
\lambda^{-2m + 2\alpha \sum_{i=1}^{k-1}n_{i+1}+ \big( m-(n_k - \alpha n_{k+1}) \big)} \ , \quad  &n_k - \alpha n_{k+1} \leq m \leq n_k + \alpha n_{k+1} \\[10pt]
\lambda^{-2m + 2\alpha \sum_{i=1}^{k-1}n_{i+1}} \ , \quad  &n_k + \alpha n_{k+1} \leq m \leq n_{k+1} - \alpha n_{k+2}
\end{cases}
\]
It is a standard procedure to check that $\mu$ is a well defined, positive and finite measure on $\overline{\mathcal{U}}(\alpha)$. Therefore, (see for example \cite{Falconer techniques}) we can get a lower bound for the dimension by calculating the lower local dimension of the measure $\mu$ for some $\underline{x} \in \overline{\mathcal{U}}(\alpha)$
\begin{align*}
\underline{\dim}_{H} \mu (\underline{x}) \ 
&= \ 
\liminf_{m\to +\infty} \frac{\log \mu (C_m(\underline{x}))}{\log \lambda^{-m}} \ 
= \ 
\lim_{k\to +\infty} \frac{\log \mu (C_{n_{k}+\alpha n_{k+1}}(\underline{x}))}{\log \lambda^{-(n_{k}+\alpha n_{k+1})}}  
\\[10pt]&= \ 
\lim_{k\to +\infty} \frac{\log \lambda^{-2(n_{k}+\alpha n_{k+1}) + 2\alpha \sum_{i=1}^{k}n_{i+1}}}{\log \lambda^{-(n_{k}+\alpha n_{k+1})}} \ 
= \ 
\lim_{k\to +\infty} \frac{2(n_{k}+\alpha n_{k+1}) - 2\alpha \sum_{i=1}^{k}n_{i+1}}{n_{k}+\alpha n_{k+1}} 
\\[10pt]&= \ 
2 - 2\alpha \lim_{k\to +\infty} \frac{\sum_{i=1}^{k} \theta^{i+1}}{\theta^k + \alpha \theta^{k+1}} \ 
= \ 
2 - 2\alpha \frac{\theta^2}{(1+\alpha \theta) (\theta -1)}.
\end{align*}
Now by maximizing the lower local dimension we will get the best lower bound (from this procedure). This occurs when the function $\theta \mapsto \frac{\theta^2}{(1+\alpha \theta) (\theta -1)}$ minimizes in $(1, 1/\alpha)$. It is not hard to see that this happens for $\theta_0 = \frac{2}{1-a} \in \big(1, 1/\alpha)$ and it takes the value $\frac{4}{(1+\alpha)^2}$. In fact $\theta_0$ is indeed appropriate so that \eqref{eq. case 1'} also holds. Therefore, for this $\theta_0$
\[
\underline{\dim}_{H} \mu (\underline{x}) \ = \ 2 - 2\alpha \frac{4}{(1+\alpha)^2} \ = \ 2\frac{(1-\alpha)^2}{(1+\alpha)^2}
\]
which implies that
\[
\dim_{H} \big( \overline{\mathcal{U}}(\alpha) \big) \ \geq \ 2\frac{(1-\alpha)^2}{(1+\alpha)^2} .
\]

\begin{figure} [h] 
\begin{tikzpicture}
\path (-4,0) --node{$\cdots$} (-2,0);
\draw (-2.6,0) -- (-2.2,0);
\draw (-2.3,0) node[anchor=north, outer sep=0.2cm] {\footnotesize $-1$};
\draw (-1.8,0) node[anchor=north, outer sep=0.2cm] {\footnotesize $0$};
\draw (-1.3,0) node[anchor=north, outer sep=0.2cm] {\footnotesize $1$};
\draw[|-] (-2.2,0)--(-1.8,0);
\draw[|-] (-1.8,0)--(-1.3,0);
\draw[|-] (-1.3,0)--(0.5,0);
\draw[blue, very thick, |-] (0.5,0) -- (2,0);
\draw[blue, very thick, |-|] (2,0) -- (3.5,0);
\draw (0.5,0) node[anchor = north, outer sep=0.2cm] {\footnotesize $n_{k}-\alpha n_{k+1}$};
\draw (2,0) node[anchor = north, outer sep=0.2cm] {\footnotesize $n_{k}$};
\draw (3.5,0) node[anchor = north, outer sep=0.2cm] {\footnotesize $n_{k} + \alpha n_{k+1}$};
\draw (3.5,0) -- (12,0);
\draw[blue, very thick, |-] (7,0)--(9.2,0);
\draw[blue, very thick, |-|] (9.2,0) -- (11.4,0);
\draw (7,0) node[anchor = north, outer sep=0.2cm] {\footnotesize $n_{k+1}-\alpha n_{k+2}$};
\draw (9.3,0) node[anchor = north, outer sep=0.2cm] {\footnotesize $n_{k+1}$};
\draw (11.4,0) node[anchor = north, outer sep=0.2cm] {\footnotesize $n_{k+1} + \alpha n_{k+2}$};
\path (11.8,0) --node{$\cdots$} (12.8,0);
\end{tikzpicture}
\caption{Example for the positions of the fixed blocks for a $1<\theta <1/\alpha$, given that $0\leq \ \alpha \ \leq 3-2\sqrt{2}$ (no overlaps).}
\label{Figure  no overlap}
\end{figure}
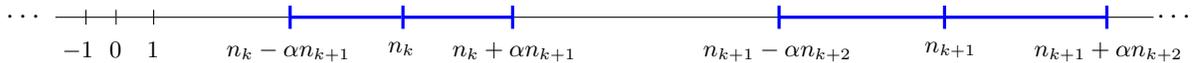

%%%%with overlaps%%%%%
We assume now that
\begin{equation} \label{eq. case 2}
n_k + \alpha n_{k+1} >  n_{k+1} - \alpha n_{k+2} \ , \quad \text{for all } k\in \mathbb{N}.
\end{equation}
This implies that there are intersections of the fixed blocks that occur (see Figure~\ref{Figure  with overlaps}). 
Now relation \eqref{eq. case 2} implies that
\begin{equation} \tag{3$'$} \label{eq. case 2'}
1+\alpha \theta > \theta - \alpha \theta^2.
\end{equation}
We can find a $\theta$ that satisfies relation \eqref{eq. case 2'} given that $\alpha > 3-2\sqrt{2}$. In fact it is satisfied for all $\theta >1$ as long as  $\alpha > 3-2\sqrt{2}$. This condition, completely fixes all the digits the in the right-hand side, except for a finite starting block from position $0$ to position $N_1=\theta$, which does not contribute to the dimension. 

\begin{figure}[H] 
\begin{tikzpicture}
\path (-4,0) --node{$\cdots$} (-2,0);
\draw (-2.6,0) -- (-2.2,0);
\draw (-2.3,0) node[anchor=north, outer sep=0.2cm] {\footnotesize $-1$};
\draw (-1.8,0) node[anchor=north, outer sep=0.2cm] {\footnotesize $0$};
\draw (-1.3,0) node[anchor=north, outer sep=0.2cm] {\footnotesize $1$};
\draw[|-] (-2.2,0)--(-1.8,0);
\draw[|-] (-1.8,0)--(-1.3,0);
\draw[|-] (-1.3,0)--(12,0);
\draw[blue, thick, |-] (-0.5,0.07) -- (1,0.07);
\draw[blue, thick, |-|] (1,0.07) -- (2.5,0.07);

\draw[blue, dashed] (-0.485, 0) -- (-0.485,0.6);
\draw[blue, dashed] (1.015, 0) -- (1.015,0.6);
\draw[blue, dashed] (2.485, 0) -- (2.485,0.6);
\draw (-0.5,0) node[anchor = south, outer sep=0.5cm] {\footnotesize \textcolor{blue}{$n_{k}-\alpha n_{k+1}$}};
\draw (1,0) node[anchor = south, outer sep=0.5cm] {\footnotesize \textcolor{blue}{$n_{k}$}};
\draw (2.5,0) node[anchor = south, outer sep=0.5cm] {\footnotesize \textcolor{blue}{$n_{k} + \alpha n_{k+1}$}};
\draw (6.4,0) -- (12,0);
%\draw[red, very thick] (2.02,-0) -- (2.02,-0.3);
%\draw[red, very thick] (4.22,0) -- (4.22,-0.3);
%\draw[red, very thick] (6.38,-0) -- (6.38,-0.3);
\draw[red, thick, |-] (2,-0.07)--(4.2,-0.07);
\draw[red, thick, |-|] (4.2,-0.07) -- (6.4,-0.07);

\draw[red, dashed] (2.015, 0) -- (2.015,-0.6);
\draw[red, dashed] (4.215, 0) -- (4.215,-0.6);
\draw[red, dashed] (6.385, 0) -- (6.385,-0.6);
\draw (2,0) node[anchor = north, outer sep=0.5cm] {\footnotesize \textcolor{red}{$n_{k+1}-\alpha n_{k+2}$}};
\draw (4.4,0) node[anchor = north, outer sep=0.5cm] {\footnotesize \textcolor{red}{$n_{k+1}$}};
\draw (6.4,0) node[anchor = north, outer sep=0.5cm] {\footnotesize \textcolor{red}{$n_{k+1} + \alpha n_{k+2}$}};
\path (11.8,0) --node{$\cdots$} (12.8,0);
\end{tikzpicture}

\hfill\\

\begin{tikzpicture}
\path (-4,0) --node{$\cdots$} (-2,0);
\draw (-2.6,0) -- (-2.2,0);
\draw (-2.3,0) node[anchor=north, outer sep=0.2cm] {\footnotesize $-1$};
\draw (-1.8,0) node[anchor=north, outer sep=0.2cm] {\footnotesize $0$};
\draw (-1.3,0) node[anchor=north, outer sep=0.2cm] {\footnotesize $1$};
\draw[|-] (-2.2,0)--(-1.8,0);
\draw[|-] (-1.8,0)--(-1.3,0);
\draw[|-] (-1.3,0)--(12,0);
\draw[blue, thick, |-] (-0.5,0.07) -- (1,0.07);
\draw[blue, thick, |-|] (1,0.07) -- (2.5,0.07);
\draw (1,0) node[anchor = north, outer sep=0.2cm] {\footnotesize \textcolor{blue}{$n_{k}$}};

\draw (1.5,0) -- (12,0);
\draw[red, thick, |-] (2,-0.07)--(4.2,-0.07);
\draw[red, thick, |-|] (4.2,-0.07) -- (6.4,-0.07);
\draw (4.4,0) node[anchor = north, outer sep=0.2cm] {\footnotesize \textcolor{red}{$n_{k+1}$}};

\draw[teal, thick, |-] (2.25,-0.15)--(7,-0.15);
\draw[teal, thick, |-|] (7,-0.15) -- (11.75,-0.15);
\draw (7.2,0) node[anchor = north, outer sep=0.2cm] {\footnotesize \textcolor{teal}{$n_{k+2}$}};

\path (11.8,0) --node{$\cdots$} (12.8,0);
\end{tikzpicture}
\caption{Examples for the positions of the fixed blocks for a $1<\theta <1/\alpha$, given that $\alpha \ \geq 3-2\sqrt{2}$. The first one depicts single overlapping, where the second one depicts double overlapping \big(see also \eqref{eq. condition for centres, for case 2} and \eqref{eq. condition about the left ends of consecutive fixed blocks}\big).}
\label{Figure  with overlaps}
\end{figure}
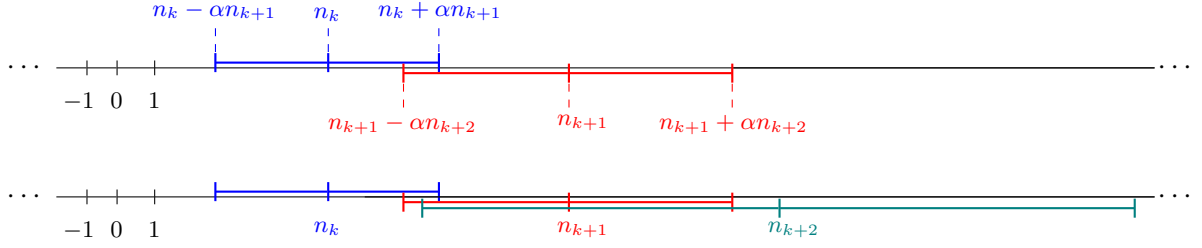

Furthermore, it creates some fixed blocks in the left-hand side. Indeed, let us assume that $d_{\Sigma}(\sigma^{n_{k}}(\underline{x}), \underline{x})< \lambda^{-\alpha n_{k+1}}$ and $d_{\Sigma}(\sigma^{n_{k+1}}(\underline{x}), \underline{x})< \lambda^{-\alpha n_{k+2}}$. Then $x_{n + i} = x_i$ for $|i| \leq \alpha n_{k+1}$ and $x_{n_{k+1} + j} = x_j$ for $|j| \leq \alpha n_{k+2}$. Assuming that the overlapping condition given by \eqref{eq. case 2} also holds, we have that $x_{-\alpha n_{k+2} + \ell} = x_{n_{k+1}-\alpha n_{k+2} + \ell} = x_{n_{k+1} - n_k - \alpha n_{k+2} + \ell}$ for all $\ell \in [0, n_k + \alpha n_{k+1} - n_{k+1} + \alpha n_{k+2}]$. In particular, this means that the digits of $\underline{x}$ from the position $-\alpha n_{k+2}$ to $n_k+(\alpha -1)n_{k+1}$, if $n_k+(\alpha -1)n_{k+1}<0$ or from the position  $-\alpha n_{k+2}$ to $0$, otherwise, are fixed (see also Figure~\ref{Figure how the left fixed blocks appear}). 

Restating what was discussed above, for each $k$, the condition
\[
 d_{\Sigma}(\sigma^{n_{k+1}}(\underline{x}), \underline{x})< \lambda^{-\alpha n_{k+2}}
\]
creates (in the left-hand side) the fixed block
\begin{align}
[- \alpha n_{k+2}, \ldots ,  n_k +(\alpha -1)n_{k+1}] \ , \quad \text{ if } n_{k+1} > n_k + \alpha n_{k+1} \label{eq. condition for centres, for case 2} \\
[- \alpha n_{k+2}, \ldots ,  0] \ , \quad \text{ if } n_{k+1} \leq n_k + \alpha n_{k+1}.
\end{align}
This means that different conditions on where the centres of each block land (in the right-hand side) with respect to the previous block gives different fixed blocks in the left-hand side. Of course the condition \eqref{eq. condition for centres, for case 2} is the only one of the two that does not fixes all the digits in the left-hand side as well. So we need each of the centres of the fixed block to land outside of the immediately preceding fixed block (as it is depicted in Figure~\ref{Figure  with overlaps}) and it gives
\begin{equation} \tag{4$'$} \label{eq. condition' for centres, for case 2}
\theta> 1+\alpha \theta \ \Rightarrow \ \theta > \frac{1}{1- \alpha}.
\end{equation}
Condition \eqref{eq. condition' for centres, for case 2} along with the initial condition \eqref{eq. condition 1' (for all alpha)} imposes some further restriction on $\alpha$. Namely, we need
\[
\frac{1}{\alpha} > \frac{1}{1- \alpha}
\]
which can be satisfied only for $\alpha< 1/2$. Therefore we have
\[
3-2\sqrt{2} < \alpha < \frac{1}{2}.
\]
Lastly, it is not very hard to see that $\theta \in (1, 1/\alpha )$ also implies that 
\begin{equation} \label{eq. condition about the left ends of consecutive fixed blocks}
n_{k}-\alpha n_{k+1} < n_{k+1} - \alpha n_{k+2} \ , \quad \text{for all } k\in \mathbb{N}
\end{equation}
which is in line with the second figure in Figure~\ref{Figure  with overlaps}.

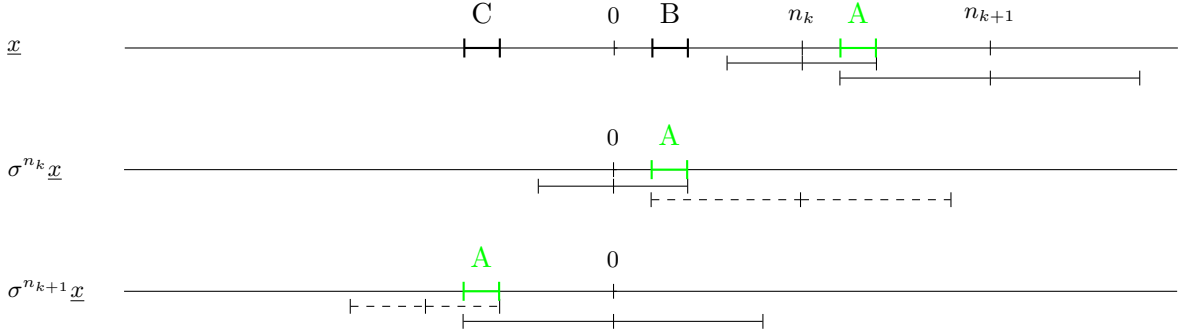
\begin{figure}[H] 
\begin{tikzpicture}
\draw (-4.05,0) node[anchor=east, outer sep=0.2cm] {\footnotesize $\underline{x}$};
\draw (-3,0)--(11,0);

\draw[|-] (3.5,0) -- (3.55,0);
\draw (3.5,0) node[anchor=south, outer sep=0.2cm] {\footnotesize $0$};

\draw [|-] (5,-0.2) -- (6,-0.2);
\draw [|-|] (6,-0.2) -- (7,-0.2);
\draw [|-] (6,0) -- (6.01,0);
\draw (6,0) node[anchor=south, outer sep=0.2cm] {\footnotesize $n_{k}$};

\draw [|-] (6.5,-0.4) -- (8.5,-0.4);
\draw [|-|] (8.5,-0.4) -- (10.5,-0.4);
\draw [|-] (8.5,0) -- (8.51,0);
\draw (8.5,0) node[anchor=south, outer sep=0.2cm] {\footnotesize $n_{k+1}$};

\draw[green, thick, |-|] (6.5,0)--(7,0);
\draw (6.75,0) node[anchor=south, outer sep=0.2cm] {\textcolor{green}{A}};
\draw[thick, |-|] (4,0)--(4.5,0);
\draw (4.25,0) node[anchor=south, outer sep=0.2cm] {B};
\draw[thick, |-|] (1.5,0)--(2,0);
\draw (1.75,0) node[anchor=south, outer sep=0.2cm] {C};
\end{tikzpicture}

\hfill\\

\begin{tikzpicture}
\draw (-3.5,0) node[anchor=east, outer sep=0.2cm] {\footnotesize $\sigma^{n_k}\underline{x}$};
\draw (-3,0)--(11,0);

\draw[|-] (3.5,0) -- (3.55,0);
\draw (3.5,0) node[anchor=south, outer sep=0.2cm] {\footnotesize $0$};

\draw [|-] (2.5,-0.22) -- (3.5,-0.22);
\draw [|-|] (3.5,-0.22) -- (4.5,-0.22);

\draw [dashed, |-|] (4,-0.4) -- (6,-0.4);
\draw [dashed, -|] (6,-0.4) -- (8,-0.4);

\draw[green, thick, |-|] (4,0)--(4.5,0);
\draw (4.25,0) node[anchor=south, outer sep=0.2cm] {\textcolor{green}{A}};
\end{tikzpicture}

\hfill\\

\begin{tikzpicture}
\draw (-3.2,0) node[anchor=east, outer sep=0.2cm] {\footnotesize $\sigma^{n_{k+1}}\underline{x}$};
\draw (-3,0)--(11,0);

\draw[|-] (3.5,0) -- (3.55,0);
\draw (3.5,0) node[anchor=south, outer sep=0.2cm] {\footnotesize $0$};

\draw [dashed, |-] (0,-0.2) -- (1,-0.2);
\draw [dashed, |-|] (1,-0.2) -- (2,-0.2);

\draw [|-] (1.5,-0.4) -- (3.5,-0.4);
\draw [|-|] (3.5,-0.4) -- (5.5,-0.4);

\draw[green, thick, |-|] (1.5,0)--(2,0);
\draw (1.75,0) node[anchor=south, outer sep=0.2cm] {\textcolor{green}{A}};

\end{tikzpicture}
\caption{A simple depiction of how the fixed blocks in the left-hand side appear: Block A (green) denotes the block we get by the overlapping condition and how it "slides" as we apply the shift operator. The condition $ d_{\Sigma}(\sigma^{n_{k}}(\underline{x}), \underline{x})< \lambda^{-\alpha n_{k+1}}$ forces the digits in block A to be equal with the digits in block B. The condition $d_{\Sigma}(\sigma^{n_{k+1}}(\underline{x}), \underline{x})< \lambda^{-\alpha n_{k+2}}$ forces the digits block A to be equal with the digits in block C. Since block A is already predetermined though by block B, block C is a fixed block too.}
\label{Figure how the left fixed blocks appear}
\end{figure}

In order to secure some freedom---even with all the previous restrictions---we need some further conditions. Namely, we need the fixed blocks in the left-hand side to be disjoint. More precisely, the blocks that are created by consecutive iterations---i.e.\ by $\sigma^{n_{k+1}}$ and $\sigma^{n_{k+2}}$---need to be disjoint. Then the two blocks that are created are
\[
[- \alpha n_{k+3}, \ldots , n_{k+1} +(\alpha -1)n_{k+2}] \quad \text{ and } \quad [- \alpha n_{k+2}, \ldots , n_k +(\alpha -1)n_{k+1}] 
\]
and since $- \alpha n_{k+3} < - \alpha n_{k+2}$, the only way to ensure disjointness (see Figure~\ref{Figure  fixed blocks left-hand side}) is if
\begin{equation} \label{eq. create disjointness in the left side}
n_{k+1} +(\alpha -1)n_{k+2} < - \alpha n_{k+2}
\end{equation}
which implies, using the fact that $\alpha < \frac{1}{2}$, that
\begin{equation} \tag{7$'$} \label{eq. create' disjointness in the left side}
\theta > \frac{1}{1-2\alpha}.
\end{equation}
Combining this with the initial condition on $\theta$ we need the following inequality to hold
\[
\frac{1}{\alpha} > \frac{1}{1-2\alpha}.
\]
This can only occur for $\alpha<1/3$. So in the end we need
\[
3-2\sqrt{2} \ < \ \alpha \ < \ \frac{1}{3}.
\]

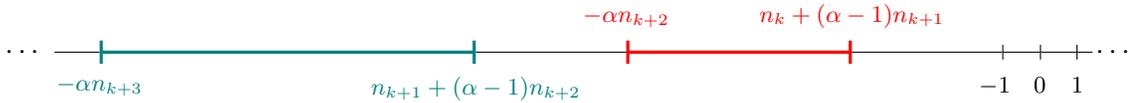
\begin{figure} [h] 
\begin{tikzpicture}
\path (-4,0) --node{$\cdots$} (-2,0);
\draw (-2.6,0) -- (-2,0);

\draw[teal, very thick, |-|] (-2,0) -- (3,0);
\draw (-2,0) node[anchor = north, outer sep=0.2cm] {\footnotesize \textcolor{teal}{$- \alpha n_{k+3}$}};
\draw (3,0) node[anchor = north, outer sep=0.2cm] {\footnotesize \textcolor{teal}{$n_{k+1} +(\alpha -1) n_{k+2}$}};

\draw (3,0) -- (9,0);
\draw[red, very thick, |-|] (5,0) -- (8,0);
\draw (5,0) node[anchor = south, outer sep=0.2cm] {\footnotesize \textcolor{red}{$- \alpha n_{k+2}$}};
\draw (8,0) node[anchor = south, outer sep=0.2cm] {\footnotesize \textcolor{red}{$n_{k} +(\alpha -1) n_{k+1}$}};

\draw (9,0) -- (10,0);
\draw (11,0) -- (11.2,0);
\draw[|-] (10,0) -- (10.5,0);
\draw[|-|] (10.5,0) -- (11,0);
\draw (9.9,0) node[anchor=north, outer sep=0.2cm] {\footnotesize $-1$};
\draw (10.5,0) node[anchor=north, outer sep=0.2cm] {\footnotesize $0$};
\draw (11,0) node[anchor=north, outer sep=0.2cm] {\footnotesize $1$};

\path (11,0) --node{$\cdots$} (12,0);
\end{tikzpicture}
\caption{Example for the positions of the fixed blocks in the left-hand side, given we have ensured disjointness.}
\label{Figure  fixed blocks left-hand side}
\end{figure}

Let 
\[
3-2\sqrt{2} < \alpha  < \frac{1}{3} \qquad \text{ and } \qquad \frac{1}{1-2\alpha}<\theta<\frac{1}{\alpha}
\] 
and consider an
$\underline{x} \in \overline{\mathcal{U}}(\alpha)$. We define for every $k\geq 1$
\begin{align*}
&\mu \big( C_m(\underline{x}) \big) 
:=
\\&= 		
\begin{cases}
\lambda^{-m + \sum_{i=1}^{k} \big(n_i + (\alpha-1)n_{i+1} + \alpha n_{i+2}\big)} \ , \  &\alpha n_{k+2} \leq \ m \ \leq (1-\alpha )n_{k+2}-n_{k+1}  \\[10pt] 
\lambda^{-m + \sum_{i=1}^{k} \big( n_i + (\alpha-1)n_{i+1} + \alpha n_{i+2}\big) + \big(m + n_{k+1}+(\alpha -1)n_{k+2}) \big)} \ , \  &(1-\alpha)n_{k+2}-n_{k+1} \leq \ m \ \leq \alpha n_{k+3}
\end{cases}
\end{align*}
The main difference from the first case in the definition of the measure $\mu$ here is that there is essentially zero freedom from the right-hand side and thus it resembles the definition as if we were in the one-sided shift (which would have been the case if we were, for example, in the $1$-dimensional torus). 

Now we have
\begin{align*}
\underline{\dim}_{H} \mu (\underline{x}) \ 
&= \ 
\liminf_{m\to +\infty} \frac{\log \mu (C_m(\underline{x}))}{\log \lambda^{-m}} \ 
= \ 
\lim_{k\to +\infty} \frac{\log \mu (C_{\alpha n_{k+2}}(\underline{x}))}{\log \lambda^{-\alpha n_{k+2}}}  
\\[10pt]&= \ 
\lim_{k\to +\infty} \frac{\log \lambda^{-\alpha n_{k+2} + \sum_{i=1}^{k} \big(n_i + (\alpha-1)n_{i+1} + \alpha n_{i+2}\big)}}{\log \lambda^{-\alpha n_{k+2}}} \ 
\\[10pt]&= \ 
\lim_{k\to +\infty} \frac{\alpha n_{k+2} - \sum_{i=1}^{k} \big(n_i + (\alpha-1)n_{i+1} + \alpha n_{i+2}\big)}{\alpha n_{k+2}} 
\\[10pt]&= \ 
1 - \lim_{k\to +\infty} \frac{\sum_{i=1}^{k} \big(\theta^i + (\alpha-1)\theta^{i+1} + \alpha \theta^{i+2}\big)}{\alpha \theta^{k+2}} \ 
\\[10pt]&= \ 
1 - \lim_{k\to +\infty} \frac{\big(1 + (\alpha-1)\theta + \alpha \theta^{2}\big) \sum_{i=1}^{k} \theta^i }{\alpha \theta^{k+2}} \
\\[10pt]&= \
1 - \frac{\alpha \theta^2 +(\alpha -1) \theta +1}{\alpha \theta (\theta -1)}  \ 
= \  
\frac{(1- 2\alpha) \theta -1}{\alpha \theta (\theta -1)}.
\end{align*}
Via standard calculus methods, one can see that this maximizes at $\theta_1 = \frac{1}{1-\sqrt{2\alpha}}$. Of course $\theta_1$ must satisfy the inequalities \eqref{eq. condition 1' (for all alpha)} and \eqref{eq. create' disjointness in the left side}, i.e.\ we need to have
\begin{equation}
\frac{1}{1-2\alpha} < \frac{1}{1-\sqrt{2\alpha}} < \frac{1}{\alpha}.
\end{equation}
While the first inequality is true for all $0< \alpha <1/2$, for the second one, we (seemingly) have some further restrictions on $\alpha$ since it yields that $\alpha$ should be less than $2-\sqrt{3}$ (which is smaller than $1/3$). Thus, in order to maximize the lower local dimension of $\mu$ we need to consider the two cases.
Firstly we assume that $3-2\sqrt{2} \ < \ \alpha \ < \ 2-\sqrt{3}$. For this case, we have that $\theta_1$ is allowed, i.e.\ satisfies 
\[
\frac{1}{1-2\alpha} < \theta_1 < \frac{1}{\alpha}
\]
and the local dimension takes the (maximum) value
\[
\frac{(1-\sqrt{2\alpha})^2}{\alpha}
\]
and therefore,
\[
\dim_{H}\big( \overline{\mathcal{U}}(\alpha) \big) \ \geq \ \frac{(1-\sqrt{2\alpha})^2}{\alpha}.
\]
We now consider the remaining possible values for $\alpha$, i.e.\ we assume that $2-\sqrt{3} \ \leq \ \alpha \ < \ \frac{1}{3}$. We have to treat this case differently, since now the maximizing $\theta_1$ is not permitted, i.e.
\[
\frac{1}{1-2\alpha} \ < \ \theta \ < \ \frac{1}{\alpha} \ < \ \frac{1}{1-\sqrt{2\alpha}} =:\theta_1.
\]
Again, with standard calculus methods, we can show that 
\[
\theta \mapsto \frac{(1- 2\alpha) \theta -1}{\alpha \theta (\theta -1)} 
\]
is strictly increasing in the interval $[\frac{1}{1-2\alpha} , \frac{1}{\alpha}]$. Hence the maximum value is $\frac{1-3\alpha}{1-\alpha}$ and it is achieved at $\frac{1}{\alpha}$. Therefore, 
\[
\underline{\dim}_{H} \mu (\underline{x}) \ \geq \ \frac{1-3\alpha}{1-\alpha} -\epsilon
\]
for all $\epsilon>0$, which implies that
\[
\dim_{H}\big( \overline{\mathcal{U}}(\alpha) \big) \ \geq \ \frac{1-3\alpha}{1-\alpha}
\]
and that completes the proof.
\end{proof}

\begin{remark}
Observe that for each of the relations 
\eqref{eq. condition 1 (for all alpha)}, 
\eqref{eq. case 1}, \eqref{eq. case 2}, 
\eqref{eq. condition for centres, for case 2}, 
\eqref{eq. condition about the left ends of consecutive fixed blocks} 
and 
\eqref{eq. create disjointness in the left side} 
it is sufficient to find just one $k_0$ for which they are true and that immediately implies that they are true for all $k\geq 1$.
\end{remark}

\section{Upper Bound} \label{Section upper-isomorphism}
In this section we are going to prove that
\begin{equation} \tag {II} \label{Equation II: lower bound}
    \dim_{H} \big( \mathcal{U}'(\alpha) \big) 
    \leq  \ 
		\begin{cases}
			2 \frac{(1-\alpha)^2}{(1+\alpha)^2} \ , \quad &0\leq \ \alpha \ \leq 3-2\sqrt{2}
			\\[12pt]
			\frac{(1-\sqrt{2\alpha})^2}{\alpha} \ , \quad &3-2\sqrt{2} \leq \ \alpha \ \leq 2-\sqrt{3} 				\\[10pt]
			\frac{1-3\alpha}{1-\alpha} \ , \quad  &2-\sqrt{3} \leq \ \alpha \ \leq 1/3 
			\\[10pt]
			0 \ , \quad &\alpha \geq 1/3 
		\end{cases}
\end{equation}

	The proof revolves around finding an optimal cover of $\mathcal{U}'(\alpha)$.
	This is done by carefully estimating the number of free digits of $\underline{x}$
	and obtaining a cover using cylinder sets. In contrast to the proof of
	the lower bound, we can no longer only consider the collection of
	$\underline{x}$ with fixed blocks with midpoints $n_k = \theta^{k}$; however, we do
	find that the $\underline{x}$ that determine the dimension of $\mathcal{U}'(\alpha)$ are
	exactly those for which the fixed blocks are distributed along such a
	sequence.

	The proof of the theorem is based on the proof of Theorem~1.2 in
	Bugeaud and Liao \cite{Bugeaud Liao}. We make a modification to
	account for an error in their proof. Although the conclusion of the
	theorem is correct, the authors pass to a subsequence, yet keep using
	the relationship between consecutive terms of the original sequence.
	This occurs at Equation~2.6 in their paper. We do not pass to a
	subsequence and instead work with the original sequence and a
	subsequence at the same time. The same modification can be applied to
	Bugeaud and Liao's paper, correcting the error.

	Again, for notational simplicity, we omit rounding certain quantities to the
	nearest integer, e.g.\ $C_{\alpha N}$ instead of
	$C_{\lceil \alpha N \rceil}$. The proofs essentially remain unchanged
	were one not to do this. Furthermore, as the fixed blocks, and in particular their ends, are not specifically predetermined as in the proof of ~\eqref{Equation I: lower bound}, it is more convenient to work with their lengths instead.

\begin{terminology}
	For $\underline{x} \in \Sigma_\Gamma$, we call the digits $x_1,x_2,\dotsc$ the
	`right digits' of $\underline{x}$, and the digits $x_{-1}, x_{-2}, \dotsc$ the
	`left digits' of $\underline{x}$.
	The condition $d(\sigma^{n}\underline{x},\underline{x}) < \lambda^{-\alpha N}$ implies that
	$x_{n + k} = x_k$ for $|k| \leq \alpha N$. We refer to the digits of
	$\underline{x}$ from $n - \alpha N$ to $n + \alpha N$ as `fixed digits' of $\underline{x}$
	and their collection as the `fixed block with midpoint $n$' of $\underline{x}$.
	The digits in between the fixed blocks of $\underline{x}$ are called `free digits'
	and form the `free blocks' of $\underline{x}$. We refer to the number of digits in
	a fixed block of $\underline{x}$ as the `length' of the fixed block. 
	
	If two fixed blocks overlap, then a fixed block consisting of left digits of $\underline{x}$
	arises. To see this, suppose that $d(\sigma^{n}\underline{x}, \underline{x}) < \lambda^{-\alpha N}$ and
	$d(\sigma^{m}\underline{x}, \underline{x}) < \lambda^{-\alpha M}$ with $n < m$ and
	$m - \alpha M < n + \alpha N$. For simplicity, suppose also that
	neither fixed block contains the other, i.e., that
	$n - \alpha N < m - \alpha M$ and $n + \alpha N < m + \alpha M$. Then
	$x_{n + k} = x_k$ for $|k| \leq \alpha N$ and $x_{m + k} = x_k$ for
	$|k| \leq \alpha M$. In particular,
	$x_{-\alpha M + k} = x_{m-\alpha M + k} = x_{m - n - \alpha M + k}$
	for
	$k \in [0, n + \alpha N - m + \alpha M] \subset [-\alpha N, \alpha N]$,
	which can be restated as there being a fixed block consisting of the
	digits of $\underline{x}$ from $-\alpha M$ to $n - m + \alpha N$. We refer to this fixed block as a `fixed block on the left.'
\end{terminology}

\begin{proof}[Proof of~\eqref{Equation II: lower bound}]
It is clear that the result holds for $\alpha = 0$. Additionally, if
the statement holds for $\alpha \leq 1/3$, then using the fact that
$\mathcal{U}'(\alpha) \subset \mathcal{U}'(\beta)$ for $\beta \leq \alpha$, we have
that $\dim_{H} \big( \mathcal{U}'(\alpha) \big) = 0$ for $\alpha \geq 1/3$. Thus, it
suffices to restrict our attention to $\alpha \in (0, 1/3]$.

We proceed by finding an optimal cover of $\mathcal{U}'(\alpha)$ in order to
estimate the Hausdorff dimension.
Fix $\alpha \in (0,1/3]$ and let $\varepsilon > 0$. We assume that
$\varepsilon$ is sufficiently small for our purposes.
For each $\underline{x} \in \mathcal{U}'(\alpha)$ we construct a sequence $(n_k)_k$,
where $n_k = n_k(\underline{x})$ satisfies
$d(\sigma^{n_k}\underline{x},\underline{x}) < \lambda^{-\alpha n_{k+1}}$. This is done by assigning to each $N \geq M(\underline{x})$ the longest
fixed block with centre not greater than $N$ and then selecting a
subsequence of these blocks such that each block is necessarily longer
than the previous one:
for each $N \geq M$ let $0 \leq n_N' \leq N$ be the centre of the
longest block up until $N$, i.e.,
\[
  d(\sigma^{n_N'}\underline{x},\underline{x})
  = \min \{d(\sigma^{n} \underline{x}, \underline{x}): 0\leq n\leq N\}.
\]
Let $N_0 = M - 1$ and for $k \geq 1$ define
\[
  N_k = \max \{N \geq N_{k-1} + 1 : n_N' = n_{N_{k-1} + 1}'\}.
\]
Finally, for $k \geq 1$ let $n_k = n_{N_k}'$ and notice that
$n_{k+1} = N_k + 1$ for all $k \geq 1$. Thus, $(n_k)_k$ is an increasing sequence of
numbers such that $n_k$ is the midpoint of a fixed block. We call the
fixed block with midpoint $n_k$ `the $k$-th fixed block' and we denote
its length by $l_k = l_k(\underline{x})$. The $k$-th fixed block can therefore be
written as
\[
\left[n_k - \frac{l_k - 1}{2}, \dotsc, n_k + \frac{l_k - 1}{2}\right]
\]
and the fixed block arising from the overlap of the $k$-th and
$(k+1)$-th fixed blocks can be written as
\begin{equation}\label{eq:upper:block:left}
	\left[ -\frac{l_{k+1} - 1}{2}, \dotsc,
	n_k - n_{k+1} + \frac{l_k - 1}{2} \right].
\end{equation}
For simplicity, we call this block on the left `the $k$-th fixed block
on the left' even though it may be empty. Note that we omit the
explicit dependence of $n_k$ and $l_k$ on $\underline{x}$ when the dependence is
clear from the context. Since $d(\sigma^{n_k}\underline{x},\underline{x}) < \lambda^{-\alpha N_k}$, the
length of the fixed block satisfies $l_k > 2\alpha N_k + 1$ and we
obtain that $\liminf_{k \to \infty} l_k/n_{k+1} \geq 2 \alpha$, using
that $n_{k+1} = N_k + 1$.

The covering we construct will cover the elements $\underline{x}$ with the largest
number of free digits. This covering will automatically cover the
elements $\underline{y}$ such that $n_k(\underline{y}) = n_k(\underline{x})$, but for which
$l_k(\underline{y}) > l_k(\underline{x})$. As a result, we may assume, without loss of
generality, that $l_k \leq 2\alpha N_k + 3$, which implies that
$\limsup_{k \to \infty} l_k/n_{k+1} \leq 2\alpha$. Hence, for
sufficiently large $k$,
\begin{equation}\label{eq:upper:length:k}
  2(\alpha - \varepsilon) n_{k+1} \leq l_k
  \leq 2(\alpha + \varepsilon) n_{k+1}.
\end{equation}

We proceed by establishing inequalities involving $n_k$, $l_k$ and $k$.
Let
\[
  \theta = \limsup_{k \to \infty} \frac{n_{k+1}}{n_k}
\]
and denote by $(n_{k_i})$ a subsequence such that the limit superior is
attained. Again, we omit the explicit dependency of $\theta$ on $\underline{x}$
when it is clear from context. It follows directly that $\theta \geq
1$. We may further restrict our attention to when
\[
  \frac{1}{1-2\alpha} \leq \theta \leq \frac{1}{\alpha}
\]
as there are only countable many $\underline{x} \in \mathcal{U}'(\alpha)$ for which
$\theta(\underline{x})$ does not satisfy the above inequality and thus this
collection of $\underline{x}$ will not contribute to the Hausdorff dimension of
$\mathcal{U}'(\alpha)$.

Indeed, suppose that $\alpha \theta > 1 + \delta$ for some
$\delta > 0$. Then the left endpoint of the $k_i$-th fixed block, i.e.,
the quantity $n_{k_i} - (l_{k_i} - 1)/2$, converges to $-\infty$:
\[
  \lim_{i \to \infty} \left( n_{k_i} - \frac{l_{k_i}-1}{2} \right)
  = \lim_{i \to \infty} \left( 1 - \frac{l_{k_i}-1}{2n_{k_i}}
    \right) n_{k_i}
  = (1 - \alpha\theta) \lim_{i \to \infty} n_{k_i}
  = - \infty,
\]
since $1 - \alpha\theta < - \delta < 0$. This means that there is only
a countable number of $\underline{x} \in \mathcal{U}'(\alpha)$ for which
$\theta > \alpha^{-1}$.

Now suppose that $(1-2\alpha)\theta < 1 - \delta$ for some
$\delta \in (0,1)$. This implies that eventually all blocks on both
sides overlap, thus leaving only finitely many free digits. Indeed, the $k$-th fixed block on the left takes the form
\eqref{eq:upper:block:left}. Thus, the $k$-th and $(k+1)$-th fixed
blocks on the left overlap if
$n_{k+1} - n_{k+2} + (l_{k+1} - 1)/2 > - (l_{k+1} - 1)/2$ or
equivalently if $n_{k+1} + l_{k+1} - n_{k+2} - 1 > 0$. Now,
\[
    \liminf_{k \to \infty} \left( \frac{n_{k+1}}{n_{k+2}} +
      \frac{l_{k+1}}{n_{k+2}} - 1 - \frac{1}{n_{k+2}} \right) \ 
    \geq \  \left(\limsup_{k \to \infty} \frac{n_{k+2}}{n_{k+1}}
      \right)^{-1} + 2\alpha - 1 \ 
    = \ \theta^{-1} + 2\alpha - 1 \ 
    \geq \ \frac{\delta}{1 - \delta}
\]
and so
\[
  \begin{split}
    \liminf_{k \to \infty}(n_{k+1} + l_{k+1} &- n_{k+2} - 1) \\
    &\geq \liminf_{k \to \infty} \left( \frac{n_{k+1}}{n_{k+2}}
      + \frac{l_{k+1}}{n_{k+2}} - 1 - \frac{1}{n_{k+2}} \right)
      \Bigl(\liminf_{k \to \infty} n_{k+2}\Bigr) \\
    &\geq \frac{\delta}{1 - \delta} \liminf_{k \to \infty} n_{k+2} \\
    &= \infty.
  \end{split}
\]
In particular, $n_{k+1} + l_{k+1} - n_{k+2} - 1 > 0$ for all except a
finite number of $k$, i.e., eventually every pair of consecutive fixed
blocks on the left will overlap. Therefore, there are only countably
many $\underline{x} \in \mathcal{U}'(\alpha)$ for which $\theta < (1-2\alpha)^{-1}$.

Using $\lim_{k \to \infty} l_k/n_{k+1} = 2\alpha$, we obtain that
$\lim_{i \to \infty} l_{k_i}/n_{k_i} = 2\alpha \theta$. In particular,
for sufficiently large $k_i$,
\begin{equation}\label{eq:upper:length:i}
  2(\theta\alpha - \varepsilon) n_{k_i} \leq l_{k_i}
  \leq 2(\theta\alpha + \varepsilon) n_{k_i}
\end{equation}
and
\begin{equation}\label{eq:upper:midpoint:i}
  (\theta - \varepsilon)n_{k_i} \leq n_{k_i+1}
  \leq (\theta + \varepsilon)n_{k_i}.
\end{equation}
Furthermore, for sufficiently large $k$,
\begin{equation}\label{eq:upper:midpoint:k}
  n_{k+1} \leq (\theta + \varepsilon)n_k.
\end{equation}

Before continuing, we estimate the number of fixed blocks up to the
fixed block at $n_k$. Using \eqref{eq:upper:midpoint:k}, we obtain that
$n_k \leq (\theta + \varepsilon)^{k-1}n_1$. We immediately find, using
$\theta \leq \alpha^{-1}$, that for sufficiently large $k$,
\begin{equation}\label{eq:upper:k:lower}
  k \geq C_1 \log n_k,
\end{equation}
where $C_1 > 0$ is a constant independent of $\theta$ and
$\varepsilon$.
On the other hand, using \eqref{eq:upper:midpoint:i}, we obtain that
$n_{k_i} \geq (\theta - \varepsilon)^{k_i - 1} n_{k_1}$, and so for
sufficiently large $k_i$, since $\theta \geq (1-2\alpha)^{-1}$,
\begin{equation}\label{eq:upper:ki:upper}
  k_i \leq C_2 \log n_{k_i},
\end{equation}
where $C_2 > 0$ is a constant independent of $\theta$ and
$\varepsilon$.

By restarting the numbering if necessary, we may assume that the
equations~\eqref{eq:upper:length:k}, \eqref{eq:upper:midpoint:k} and
\eqref{eq:upper:k:lower} hold for all $k \geq 1$ and that the
equations~\eqref{eq:upper:length:i}, \eqref{eq:upper:midpoint:i} and
\eqref{eq:upper:ki:upper} hold for all $i \geq 1$.

We continue with establishing an optimal covering of $\mathcal{U}'(\alpha)$.
From a finite cover
$\{ [\theta_j, \theta_j + \varepsilon] \}_{j=1}^{J}$ of
$[(1-2\alpha)^{-1}, \alpha^{-1}]$, we obtain the finite cover
$\mathcal{U}'(\alpha) \subset \bigcup_{j=1}^{J} \mathcal{U}_j'(\alpha)$, where
\[
  \mathcal{U}_j'(\alpha) = \Bigl\{\underline{x}\in\mathcal{U}'(\alpha) : \limsup_{k \to \infty}
  \frac{n_{k + 1}}{n_k} \in [\theta_j, \theta_j + \varepsilon] \Bigr\}.
\]
As a result
$\dim_{H} \big( \mathcal{U}'(\alpha) \big) \leq \max_{j=1,\dotsc,M} \dim_{H} \big( \mathcal{U}_j'(\alpha) \big)$,
and so we turn our attention to finding an upper bound of
$\dim_{H} \big( \mathcal{U}_j'(\alpha) \big)$ using an appropriate cover.

Fix $j \in [1,J]$.
Note that equations \eqref{eq:upper:k:lower} and
\eqref{eq:upper:ki:upper} hold for $\theta$ replaced by $\theta_j$ up
to a slight change of the constants $C_1, C_2$ if necessary.
We find a cover by estimating the number of free
digits of each $\underline{x} \in \mathcal{U}_j'(\alpha)$. If $\underline{y} \in \Sigma_{\Gamma}$ has fixed block with the same centres and at least of the same length as $\underline{x}$ then, $d(\sigma^{n_k(\underline{x})}\underline{y},\underline{y}) < \lambda^{-\alpha n_{k+1}(\underline{x})}$. From here we obtain cylinder sets $C_K(\underline{y})$ where $\underline{y}$ is as described before.
The number of elements of this cover is asymptotically equivalent to
$\lambda^{\text{free digits}}$.

We now proceed by covering $\mathcal{U}_j'(\alpha)$ by considering all the
different ways of arranging the fixed blocks of the elements. We then
construct two different covers. The first one is constructed using
cylinder sets covering up to and including the $k_i$-th fixed block.
The second one is constructed using cylinder sets covering the fixed
blocks on the left up to and including the fixed block arising as a
result of the overlap of the $(k_i - 1)$-th and $k_i$-th fixed blocks.

\emph{Case 1: covering $\mathcal{U}_j'(\alpha)$ by cylinder sets coinciding
with fixed blocks on the right.} We begin by bounding the number of
free digits of elements of $\mathcal{U}_j'(\alpha)$ from above. Let
$\underline{x} \in \mathcal{U}_j'(\alpha)$. The number of fixed digits up to and
including the $k_i$-th block is given by
\[
  \sum_{k=1}^{k_i} l_k.
\]
It is clear that this quantity counts the number of fixed digits when
none of the fixed blocks overlap. When two fixed blocks, say the $k$-th
and $(k+1)$-th, overlap, then the quantity $l_k + l_{k+1}$ counts the
number of fixed digits on the right as well as the fixed digits on the
left arising from the overlap. Now,
\begin{equation}\label{eq:upper:right:1}
  \begin{aligned}[b]
    \sum_{k=1}^{k_i} l_k
    &\geq 2(\alpha - \varepsilon) \sum_{k=1}^{k_i} n_{k+1} \\
    &= 2(\alpha - \varepsilon) n_{k_i}
      \sum_{k=1}^{k_i}\frac{n_{k+1}}{n_{k_i}} \\
    &\geq 2(\alpha - \varepsilon)n_{k_i}
      \biggl(\frac{n_{k_i + 1}}{n_{k_i}} + \frac{1 - (\theta +
      \varepsilon)^{-k_i + 2}}{1 - (\theta +\varepsilon)^{-1}}\biggr)\\
    &\geq 2(\alpha - \varepsilon)
      \biggl(\theta - \varepsilon + \frac{1 - (\theta
      + \varepsilon)^{-C_1 \log n_{k_i} + 2}}{1 - (\theta
      + \varepsilon)^{-1}} \biggr)n_{k_i} ,
  \end{aligned}
\end{equation}
where the first line follows from \eqref{eq:upper:length:k}, the third
line from $n_k/n_{k_i} \geq (\theta + \varepsilon)^{k-k_i}$ by using
\eqref{eq:upper:midpoint:k} and the fourth line from
\eqref{eq:upper:midpoint:i} and \eqref{eq:upper:k:lower}. Thus, for
large $n_{k_i}$ we can bound \eqref{eq:upper:right:1} from below by
\[
  2(\alpha - \varepsilon) \Bigl(\theta - \varepsilon + \frac{1}{1
  - (\theta + \varepsilon)^{-1}} - \varepsilon \Bigr) n_{k_i} \ 
  \geq \ 
  2(\alpha - \varepsilon) \Bigl(\theta + \frac{\theta}{\theta
  + \varepsilon - 1} - 2\varepsilon \Bigr) n_{k_i}.
\]
This, in combination with the fact that
$\theta \in [\theta_j, \theta_j + \varepsilon]$ for
$\underline{x} \in \mathcal{U}_j'(\alpha)$, gives us
\begin{equation}\label{eq:upper:right:2}
  \begin{split}
    \sum_{k=1}^{k_i} l_k
    &\geq 2(\alpha - \varepsilon) \Bigl(\theta + \frac{\theta}
      {\theta + \varepsilon - 1} - 2\varepsilon \Bigr) n_{k_i} \\
    &\geq 2(\alpha - \varepsilon) \Bigl(\theta_j
      + \frac{\theta_j}{\theta_j + 2\varepsilon - 1} - 2\varepsilon
      \Bigr) n_{k_i} \\
    &\geq 2(\alpha - \varepsilon) \Bigl(\theta_j
      + \frac{\theta_j}{\theta_j - 1} - C_3\varepsilon \Bigr) n_{k_i},
  \end{split}
\end{equation}
for some constant $C_3 > 0$ independent of $\theta_j$ and
$\varepsilon$.

In order to obtain a cover, we fix $n \geq 1$ and consider all 
$\underline{x} \in \mathcal{U}_j'(\alpha)$ for which there exists an $i$ such that
$n_{k_i}(\underline{x}) = n$. Of these $\underline{x}$ we consider those for which $k_i(\underline{x}) = m$
for a fixed $m \geq 1$. Notice that by \eqref{eq:upper:k:lower} and
\eqref{eq:upper:ki:upper} we require that
$C_1 \log n \leq m \leq C_2 \log n$. Finally, we consider all the
different ways of arranging the first $m = k_i(\underline{x})$ fixed blocks, which
is given by a function $\sigma$ such that $n_k(\underline{x}) = \sigma(k)$. Since
$n_k(\underline{x})$ is increasing in $k$, we have that $\sigma \colon
\{1,\dotsc,m\} \to \{1,\dotsc,n\}$ is an injective and increasing
function. Let $V(n,m,\sigma)$ be the collection of
$\underline{x} \in \mathcal{U}_j'(\alpha)$ such that there exists an $i$ for which
$n_{k_i}(\underline{x}) = n$ and such that $k_i(\underline{x}) = m$, and $n_k(\underline{x}) = \sigma(k)$
for all $k = 1,\dotsc,k_i(\underline{x})$. Using this notation, we obtain a cover
of $\mathcal{U}_j'(\alpha)$ for every $N \geq 1$:
\begin{equation}\label{eq:upper:union}
  \mathcal{U}_j'(\alpha) \subset \bigcup_{n = N}^{\infty}
  \bigcup_{m = C_1 \log n}^{C_2 \log n} \bigcup_{\sigma} V(n,m,\sigma).
\end{equation}
Using cylinder sets, we cover each $V(n,m,\sigma)$ up to and including
the $m$-th fixed block. Notice that if $\underline{x} \in V(n,m,\sigma)$, then the
endpoint of the $m$-th fixed block is at most
\[
  n_{k_i}(\underline{x}) + (l_{k_i}(\underline{x}) - 1)/2
  \leq (1 + \alpha\theta(\underline{x}) + \varepsilon) n_{k_i}(\underline{x})
  \leq (1 + \alpha\theta_j + 2\varepsilon) n.
\]
Thus, we cover $V(n,m,\sigma)$ by cylinder sets of the form
$C_{(1 + \alpha\theta_j + 2\varepsilon) n}(\underline{y})$, where $\underline{y} \in V(n,m,\sigma)$. The number of $\underline{y}$ necessary is given by the
number of free digits in the block starting from
$-(1 + \alpha\theta_j + 2\varepsilon) n$ and ending at
$(1 + \alpha\theta_j + 2\varepsilon) n$. By \eqref{eq:upper:right:2},
this quantity is at most
\[
  P(n) := 1 + 2(1 + \alpha\theta_j + 2\varepsilon)n - 2(\alpha
  - \varepsilon) \Bigl(\theta_j + \frac{\theta_j}{\theta_j - 1}
  - C_3\varepsilon \Bigr) n,
\]
provided that $N \leq n$ is sufficiently large. There are at most
$n^{m} \leq n^{C_2\log n}$ increasing functions
$\sigma \colon \{1,\dotsc,m\} \to \{1,\dotsc,n\}$. Thus the total number of
cylinder sets of the form $C_{(1 +\alpha\theta_j +2\varepsilon) n}(\underline{y})$
in the union \eqref{eq:upper:union}, which (with a standard combinatorial argument) corresponds to the number of ways we can choose a block of length equal to the total length of free positions we have, i.e.\ to $\mathcal{A}_{P(n)}$, is at most
\begin{equation*}
(C_2 \log n) \cdot n^{C_2\log n} \cdot \#\mathcal{A}_{P(n)}
\end{equation*}
By Corollary \ref{Corollary admissible blocks are almost lambda to the n}, for all $\delta>0$, and for all sufficiently large $n$ (depending on $\delta$), we have that 
\[
\#\mathcal{A}_{P(n)} \leq \lambda^{(1+\delta)P(n)}
\]
Therefore, the number of cylinders and for sufficiently large $n$, is at most
\begin{equation}\label{eq:upper:number}
(C_2 \log n) \cdot n^{C_2\log n} \cdot \lambda^{(1+\delta)P(n)}.
\end{equation}
At the same time, the diameter of the cylinder sets
$C_{(1 + \alpha\theta_j + 2\varepsilon) n}(\underline{y})$ is
$\lambda^{-(1 + \alpha\theta_j + 2\varepsilon) n}$, and so a standard
covering argument leaves us with finding $s_0 > 0$ for which the sum
\[
  \sum_{n=N}^{\infty} (C_2 \log n) \cdot n^{C_2\log n} \cdot \lambda^{(1+\delta)P(n)}
  \cdot \lambda^{-(1 + \alpha\theta_j + 2\varepsilon) n s}
\]
converges for all $s > s_0$. This $s_0$ is given by
\begin{equation}\label{eq:upper:right:s}
  \begin{split}
    s_0 = s_0(\delta) &= (1+\delta) \cdot \frac{2(1 + \alpha \theta_j - \varepsilon) - 2(\alpha
      - \varepsilon)(\theta_j + \frac{\theta_j}{\theta_j - 1}
      - C_3\varepsilon)}{1 + \alpha\theta_j + 2\varepsilon} \\
    &\leq 2 (1+\delta) \cdot \frac{(1-\alpha)\theta_j - 1 + C'\varepsilon}{(1
      + \alpha\theta_j + 2\varepsilon)(\theta_j - 1)} \\
    &\leq 2 (1+\delta)  \cdot \bigg(\frac{(1-\alpha)\theta_j - 1}{(1 + \alpha\theta_j)
      (\theta_j - 1)} + C\varepsilon \bigg),
  \end{split}
\end{equation}
where $C, C' > 0$ are constants independent of $\theta_j$ and
$\varepsilon$. Thus,
\[
  \dim_{H} \big(  \mathcal{U}_j'(\alpha) \big) \leq 2 (1+\delta) \cdot \frac{(1-\alpha)\theta_j - 1}
  {(1 + \alpha\theta_j)(\theta_j - 1)} + C\varepsilon
\]
and
\[
  \begin{split}
    \dim_{H} \big( \mathcal{U}'(\alpha) \big)
    &\leq \max_{1 \leq j \leq J} \big(\dim_{H} \big( \mathcal{U}_j'(\alpha) \big) \big) \\
    &\leq 2 (1+\delta) \max_{1 \leq j \leq J} \biggl(\frac{(1 - \alpha)\theta_j
      - 1}{(1 + \alpha\theta_j)(\theta_j - 1)}\biggr) + C\varepsilon\\
    &\leq 2 (1+\delta) \sup_\theta \biggl(\frac{(1 - \alpha)\theta - 1}{(1
      + \alpha\theta) (\theta - 1)}\biggr) + C\varepsilon,
  \end{split}
\]
where the supremum in the final line is taken over
$\theta \in [(1-2\alpha)^{-1}, \alpha^{-1}]$.
The supremum is attained at $\theta = \frac{2}{1-\alpha}$ when
$\alpha \in (0,1/3]$. After substituting and letting
$\varepsilon \to 0$, we obtain for $\alpha \in (0,1/3]$ that
\begin{equation*}
  \dim_{H} \big( \mathcal{U}'(\alpha) \big)
  \leq 2(1+\delta) \biggl(\frac{1-\alpha}{1+\alpha}\biggr)^2
\end{equation*}
and that for any $\delta>0$. Therefore, by taking $\delta \to 0$ we have that
\begin{equation}\label{eq:upper:right:dim}
  \dim_{H} \big( \mathcal{U}'(\alpha) \big)
  \leq 2 \biggl(\frac{1-\alpha}{1+\alpha}\biggr)^2.
\end{equation}

%%%% Case 2
\emph{Case 2: covering $\mathcal{U}_j'(\alpha)$ by cylinder sets coinciding
with fixed blocks on the left.}
Recall from \eqref{eq:upper:block:left} that if two consecutive fixed blocks of $\underline{x} \in \mathcal{U}'(\alpha)$, say the $k$-th and $(k+1)$-th blocks overlap, then a fixed block
\[
  \left[ -\frac{l_{k+1}-1}{2}, \dotsc,
  n_k - n_{k+1} + \frac{l_{k+1}-1}{2} \right]
\]
appears on the left. We call this fixed block `the $k$-th fixed block
on the left' and denote its length by $s_k$ even if the fixed block may
be empty. The number of fixed digits in the block
$[ -(l_{k_i} - 1)/2,\dotsc, (l_{k_i} - 1)/2]$ is given by
\begin{equation}\label{eq:upper:left:fixed}
  \frac{l_{k_i} - 1}{2} + \sum_{k=1}^{k_i-1} s_k,
\end{equation}
which holds even if $s_k < 0$ for some $k$.
Indeed, the case $s_k < 0$ for some $k$ accounts for under counting the
free letters on the right, since $s_k < 0$ corresponds to the $k$-th
and $(k+1)$-th fixed blocks (on the right) not overlapping and leaving
in between them a free block of length $|s_k|$.
Following the same procedure as in Case~1, we find a lower bound of the
sum. Using that
\[
  \begin{split}
    s_k &= n_k - n_{k+1} + (l_k -1)/2 + (l_{k+1}-1)/2 \\
    &\geq n_k - n_{k+1} + (\alpha - \varepsilon)n_{k+1}
      + (\alpha - \varepsilon)n_{k+2} - 1 \\
    &= n_k - (1 - \alpha + \varepsilon)n_{k+1}
      + (\alpha - \varepsilon)n_{k+2} - 1,
  \end{split}
\]
where the second line follows from \eqref{eq:upper:length:k}, we obtain
that
\[
  \begin{split}
    \sum_{k=1}^{k_i-1} s_k
    &\geq \sum_{k=1}^{k_i-1} \bigl(n_k
      - (1 - \alpha + \varepsilon)n_{k+1} + (\alpha
      - \varepsilon)n_{k+2} - 1 \bigr) \\
    &= -k_i + 1 + n_1 + (\alpha - \varepsilon)n_2 - (1 - 2\alpha + 2\varepsilon)n_{k_i} + (\alpha - \varepsilon)n_{k_i+1} + 2(\alpha - \varepsilon) \sum_{k=3}^{k_i - 1} n_k.
  \end{split}
\]
Since $k_i \leq C_2\log n_{k_i}$ and
$n_{k_i + 1} \leq (\theta - \varepsilon)n_{k_i}$,
\begin{equation}\label{eq:upper:left:1}
  \sum_{k=1}^{k_i-1} s_k \ 
  \geq \ 
  -C_2\log n_{k_i} - (1 - 2\alpha
    + 2\varepsilon)n_{k_i} 
    + (\alpha - \varepsilon)(\theta - \varepsilon)n_{k_i}
    + 2(\alpha - \varepsilon) \sum_{k=3}^{k_i - 1} n_k
\end{equation}
Using \eqref{eq:upper:right:1},
\[
  \begin{split}
    \sum_{k=3}^{k_i - 1} n_k
    &\geq n_{k_i - 1} \biggl( \frac{1 - (\theta
      + \varepsilon)^{-k_i+3}}{1-(\theta + \varepsilon)^{-1}} \biggr)\\
    &\geq n_{k_i} (\theta + \varepsilon)^{-1} \biggl( \frac{1 - (\theta
      + \varepsilon)^{-k_i+3}}{1-(\theta + \varepsilon)^{-1}} \biggr)\\
    &= n_{k_i} \biggr( \frac{1 - (\theta
      + \varepsilon)^{-k_i+3}}{\theta + \varepsilon - 1} \biggr) \\
    &\geq n_{k_i} \biggl( \frac{1 - (\theta
      + \varepsilon)^{-C_1\log n_{k_i}+3}}{\theta + \varepsilon - 1}
      \biggr),
  \end{split}
\]
where in the second line we have used \eqref{eq:upper:midpoint:i} and
in the final line we have used \eqref{eq:upper:k:lower}. Thus, using
\eqref{eq:upper:left:1}, we obtain for large $n_{k_i}$ that
\begin{equation*}
  \sum_{k=1}^{k_i - 1} s_k \ 
  \geq \ 
  -C_2\log n_{k_i}
    - (1 - 2\alpha + 2\varepsilon)n_{k_i}
    + (\alpha - \varepsilon)(\theta - \varepsilon)n_{k_i}
  + 2(\alpha - \varepsilon) \biggl( \frac{1}{\theta + \varepsilon - 1}
  - \varepsilon \biggr) n_{k_i}.
\end{equation*}
As a result, the number of fixed digits, i.e., the quantity
\eqref{eq:upper:left:fixed}, is at least
\begin{multline}\label{eq:upper:left:2}
  (\alpha\theta - \varepsilon) n_{k_i} - \frac{1}{2} - C_2\log n_{k_i}
  - (1 - 2\alpha + 2\varepsilon)n_{k_i} \\
  + (\alpha - \varepsilon)(\theta - \varepsilon)n_{k_i}
  + 2(\alpha - \varepsilon) \biggl( \frac{1}{\theta + \varepsilon - 1}
  - \varepsilon \biggr) n_{k_i},
\end{multline}
using \eqref{eq:upper:length:i}.

Just as before, we obtain the cover
\[
  \mathcal{U}_j'(\alpha) \subset \bigcup_{n = N}^{\infty}
  \bigcup_{m = C_1 \log n}^{C_2 \log n} \bigcup_{\sigma} V(n,m,\sigma),
\]
where $V(n,m,\sigma)$ is the collection of $\underline{x} \in \mathcal{U}_j'(\alpha)$
such that there exists an $i$ for which $n_{k_i}(\underline{x}) = n$ and such that
$k_i(\underline{x}) = m$, and $n_k(\underline{x}) = \sigma(k)$ for all $k = 1,\dotsc,k_i(\underline{x})$.
Using cylinder sets, we cover each $V(n,m,\sigma)$ up to and including
the $(m-1)$-th fixed block on the left. Since the endpoint of the
$(m-1)$-th left fixed block of $\underline{x} \in V(n,m,\sigma)$ satisfies
\[
  -\frac{l_{k_i}(\underline{x}) - 1}{2}
  \geq -(\alpha\theta(\underline{x}) + \varepsilon)n_{k_i}(\underline{x})
  \geq -(\alpha\theta_j + 2\varepsilon)n,
\]
we cover each $V(n,m,\sigma)$ by cylinder sets of the form
$C_{(\alpha\theta_j + 2\varepsilon)n}(\underline{y})$, where $\underline{y} \in V(n,m,\sigma)$.
The number of $\underline{y}$ necessary is given by the number of free digits in
the block
$[-(\alpha\theta_j +2\varepsilon)n,\dotsc, (\alpha\theta_j
+ 2\varepsilon)n]$, which by \eqref{eq:upper:left:2} and
$\theta \in [\theta_j, \theta_j + \varepsilon]$, is at most
\begin{equation*}
  1 + (\alpha\theta_j + 5\varepsilon)n + C_2\log n
  + (1 - 2\alpha + 2\varepsilon) n 
  - (\alpha - \varepsilon)(\theta_j - \varepsilon) n
  - 2(\alpha - \varepsilon) \biggl(
  \frac{1}{\theta_j + 2\varepsilon - 1} - \varepsilon \biggr)n.
\end{equation*}
We can bound this quantity above by
\[
  Q(n) := 1 +  \Bigl(1 - 2\alpha - \frac{2\alpha}{\theta_j
  + 2\varepsilon - 1} + C_4 \varepsilon \Bigr)n
\]
for some constant $C_4 > 0$ independent of $\theta_j$ and
$\varepsilon$. Arguing as in equation \eqref{eq:upper:number} (through \eqref{eq:upper:right:dim})
we conclude that the total number of cylinder
sets of the form $C_{(\alpha\theta_j + 2\varepsilon)n}(\underline{y})$ in the union
\eqref{eq:upper:union} is, essentially, at most 
\begin{equation} \label{eq:upper:number2}
  (C_2 \log n) \cdot n^{C_2\log n} \cdot \lambda^{Q(n)}.
\end{equation}

The diameter of the cylinder sets
$C_{(\alpha\theta_j + 2\varepsilon)n}(\underline{y})$ is 
$\lambda^{-(\alpha\theta_j + 2\varepsilon)n}$; therefore, the same
covering argument as in Case~1 leaves us with finding $s_0 > 0$ for
which the sum
\[
  \sum_{n = N}^{\infty} (C_2 \log n) \cdot n^{C_2\log n}
  \cdot \lambda^{Q(n)} \cdot \lambda^{-(\alpha\theta_j + 2\varepsilon)ns}
\]
converges for all $s > s_0$. This $s_0$ is given by
\[
  \begin{split}
    s_0 &= \frac{1 - 2\alpha - \frac{2\alpha}{\theta_j + 2\varepsilon
      - 1} + C_4\varepsilon}{\alpha\theta_j + 2\varepsilon} \\
    &\leq \frac{(1-2\alpha)\theta_j - 1 + C'\varepsilon}
      {(\alpha\theta_j + 2\varepsilon)(\theta_j + 2\varepsilon - 1)} \\
    &\leq \frac{(1-2\alpha)\theta_j - 1}{\alpha\theta_j(\theta_j - 1)}
      + C\varepsilon,
  \end{split}
\]
where $C, C' > 0$ are constants independent of $\theta_j$ and
$\varepsilon$. Hence,
\[
  \begin{split}
    \dim_{H} \big( \mathcal{U}'(\alpha) \big)
    &\leq \max_{1\leq j \leq J} \big(\dim_{H} \big( \mathcal{U}_j'(\alpha)\big) \big) \\
    &\leq \max_{1\leq j \leq J} \biggl( \frac{(1-2\alpha)\theta_j
      - 1}{\alpha\theta_j (\theta_j - 1)} \biggr) + C\varepsilon \\
    &\leq \sup_\theta \biggl( \frac{(1-2\alpha)\theta - 1}{\alpha\theta
      (\theta - 1)} \biggr) + C\varepsilon,
  \end{split}
\]
where the supremum is taken over
$\theta \in [(1-2\alpha)^{-1}, \alpha^{-1}]$.
If $\alpha \leq 2 - \sqrt{3}$, then the supremum is attained when
$\theta = \frac{1 + \sqrt{2\alpha}}{1 + 2\alpha}$, and if
$\alpha > 2 - \sqrt{3}$, then the supremum is attained when
$\theta = \alpha^{-1}$. After substituting and letting
$\varepsilon \to 0$, we obtain that
\begin{equation}\label{eq:upper:left:dim}
  \dim_{H} \big( \mathcal{U}'(\alpha) \big) \leq
  \begin{cases}
    \frac{(1-\sqrt{2\alpha})^2}{\alpha}
    & \text{if $\alpha \leq 2-\sqrt{3}$}, \\
    \frac{1-3\alpha}{1-\alpha}
    & \text{if $2-\sqrt{3} < \alpha \leq 1/3$}.
  \end{cases}
\end{equation}

Combining \eqref{eq:upper:right:dim} and \eqref{eq:upper:left:dim} and
choosing the optimal upper bound for each $\alpha$ concludes the proof.
\end{proof}

\section{Cardinality}
  Since $\mathcal{U} \left(\frac{1}{3}\right) \subset \mathbb{T}^2$ and the
cardinality of $\mathbb{T}^2$ is the continuum, it is enough to
identify a subset $A \subset \mathcal{U} \left(\frac{1}{3}\right)$ such that the
cardinality of $A$ is the continuum.  We will do this by
finding a set $B \subset \mathcal{U}' \left(\frac{1}{3}\right)$ such that
$B$ is a continuum and $A = \pi (B)$ is a continuum as well.

Let $\delta$ be an arbitrary natural number. The goal is to
construct for $\mathcal{U}' \left(\frac{1}{3}\right)$, a sequence of fixed
blocks with centres $n_k$, such that the free blocks on the left
are exactly of length $\delta-1$. This is done as follows.

Let $n_1 \in \mathbb{N}$ and define recursively,
$n_{k+1} = 3 (n_k + \delta)$. Then $\frac{1}{3} n_k \in
\mathbb{N}$ for all $k > 1$.
Around each $n_k$, we declare that the block
$[n_k - \frac{1}{3} n_{k+1}, \ldots, n_k + \frac{1}{3} n_{k+1}]
= [-\delta, \ldots, 2 n_k + \delta]$ is non-free. We let $B$ be
the set of sequences with such fixed blocks. By
construction, $B \subset \mathcal{U}' \left(\frac{1}{3}\right)$.

The fixed blocks on the right clearly overlap and cover
$[-\delta, \ldots]$. They therefore give rise to fixed
blocks on the left. The overlap of the block around $n_k$ and
$n_{k+1}$ is simply the block around $n_k$. Shifted to the
left, it is the block
\[
[-\delta - n_{k+1}, 2 n_k + \delta - n_{k+1}] = [- \delta -
n_{k+1} , - 2 \delta - n_k].
\]
It follows that the separation (number of free digits) between
two consecutive blocks on the left is
\[
-\delta - n_{k+1} - (-2\delta - n_{k+1}) - 1 = \delta - 1.
\]

The free digits are free only to the extent that they need to
fit together with the fixed blocks according to the rules of
the subshift. By taking $\delta$ sufficiently large, there will
be actual freedom in the countably many free blocks on the left
and this implies that the set $B$ is uncountable.
Finally, it is clear that $A = \pi(B)$ is uncountable as well.

\section{Dimensional relation between the manifold and the coding space} \label{Section  dim manifold = dim coding space - isomorphism}
The main result of this section is to show that there is no difference in considering the original set in the manifold or its counterpart in the coding space since we can retrieve the same dimensional results regardless. We remind that (without loss of generality), $A$ is a hyperbolic, integer matrix with $\det A=1$ and $\lambda>1$ denotes the largest eigenvalue of $A$. Furthermore, $T \colon \mathbb{T}^2 \to \mathbb{T}^2$ is defined by $T(x) = Ax \pmod{1}$, $x\in \mathbb{T}^2$. We also remind here the definitions of the two sets $\mathcal{U}(\alpha)$ and $\mathcal{U}'(\alpha)$:
\begin{align*}
&\mathcal{U}(\alpha) = \left\{x\in \mathbb{T}^2: \ \exists M=M(x) \geq 1 \text{ such that } \forall N\geq M, \ \exists n\leq N \text{ such that } d_{\mathbb{T}^2}(T^nx, x) \leq \lambda^{-\alpha N} \right\}
\\&
\mathcal{U}'(\alpha) = \left\{x\in \Sigma_{\Gamma}: \ \exists M=M(x) \geq 1 \text{ such that } \forall N\geq M, \ \exists n\leq N \text{ such that } d_{\Sigma}(\sigma^nx, x) \leq \lambda^{-\alpha N} \right\}
\end{align*}
Furthermore, we remind that we consider the shift space $\Sigma$ endowed with the metric 
\[
d_{\Sigma}(\underline{x}, \underline{y}) = \lambda^{-k(\underline{x}, \underline{y})}
\]
where $k(\underline{x}, \underline{y}) = \max \{n\in \mathbb{N}: \ x_{i}=y_{i}, \text{ for all } |i|\leq n \}$. 
Let $(X,\rho)$ be a metric space and $Y \subseteq X$. Throughout the rest of this part, by $|Y|$ we denote the diameter of the set $Y$.

\begin{remark}
In what follows, it will be showcased the reason for which we chose this specific metric, as it will give us the correct metrical correspondence between the manifold and the shift space, in order to acquire the following result in Proposition \ref{Proposition sufficient to calculate dim for the respective set in the shift space}.
\end{remark}

\begin{proposition} \label{Proposition sufficient to calculate dim for the respective set in the shift space}
Let $\alpha \geq 0$ and consider the two sets $\mathcal{U}(\alpha) \subset \mathbb{T}^2$ and $\mathcal{U}'(\alpha) \subset \Sigma_{\Gamma}$. Then 
\[
\dim_{H} \big(\mathcal{U}(\alpha)\big)
=
\dim_{H} \big(\mathcal{U}'(\alpha)\big) .
\]
\end{proposition}

Firstly, in the following result, we show that the pull back via the coding map $\pi$ of a set in the coding space does not alter its Hausdorff dimension. 
\begin{proposition} \label{Proposition dim of E in shift space  equals  dim of p(E) in manifold through coding map}
Consider $A$ to be a hyperbolic, area preserving matrix with
integer entries and the system $T\colon x\mapsto Ax \pmod{1} \colon \mathbb{T}^2 \to \mathbb{T}^2$. Furthermore let $\Sigma_{\Gamma}$ be the subshift of finite type that encodes this system and $\pi$ the respective coding map. Then
\[
\dim_{H} \big(E \big) =  \dim_{H} \big(\pi (E) \big)
\]
for all $E \subset \Sigma_{\Gamma}$.
\end{proposition}

The following result is well known and it can be found for example in \cite[Chapter 4]{Mattila}.
\begin{lemma}  \label{Lemma H. dimension with only considering open balls}
Let $X$ be a compact metric space and $Y\subset X$. 
Then it is sufficient to consider covers of open balls for $\dim_{H}(Y)$.
\end{lemma}

\begin{lemma} \label{Lemma balls=symm.cylinders in Sigma}
In the two-sided shift space $(\Sigma, d_{\Sigma})$, the open balls are exactly the symmetric cylinders $C_{[x_{-n}, \ldots , 0, \ldots , x_n]}$.
\end{lemma}
\begin{proof}
	Obviously the symmetric cylinders, $C_{[x_{-n}, \ldots , 0, \ldots , x_n]}$, are open balls and more precisely, they are balls with centre any point $\underline{y}$ for which $x_{i}=y_{i}$ and $x_{-i}=y_{-i}$, for all $i=0,\ldots n$ and radius $\lambda^{-n}$. 
	
	Now let $B(\underline{x}, r)$ to be a ball of radius $r<1$. Then there exist an $n$ such that 
	\[
	\lambda^{-n-1} \leq r < \lambda^{-n}.
	\]
	This means that $\underline{y} \in B(\underline{x}, r)$ if and only if $\underline{y} \in C_{[x_{-n}, \ldots , 0, \ldots , x_n]}$ and that completes the proof.
\end{proof}

By combining Lemma~\ref{Lemma H. dimension with only considering open balls} and Lemma~\ref{Lemma balls=symm.cylinders in Sigma} we immediately get the following.
\begin{lemma} \label{Lemma for H. dimension in shift space, sufficient cylinders}
In $\Sigma$  one needs to consider only covers of symmetric cylinders in order to calculate the Hausdorff dimension of any subset of\/ $\Sigma$.
\end{lemma}

For notation convenience, by $C_n$ we mean the symmetric cylinders $C_{[x_{-n}, \ldots , 0, \ldots , x_n]}$. These cylinders can be represented also as
\[
\bigcap_{i=-n}^{n} \sigma^{i}(C_{x_i})
\]
where $x_i$ is some letter from the alphabet and $C_{x_i}$ is the cylinder which contains all the elements which have $x_i$ at the $0$-position.
By $\widetilde{C}_n$ we mean the respective "symmetric" cylinders in the manifold,
\[
\widetilde{C}_n := \bigcap_{i=-n}^{n} T^{i}(X_{j_i})
\]
where $X_{j_i}$ is some element of the Markov partition and $T \colon \mathbb{T}^2 \to \mathbb{T}^2$ with $T(x) = Ax \pmod{1}$, for any $x\in \mathbb{T}^2$.
Observe that through the coding, $x_i = j_i
$ and
\[
\widetilde{C}_n = \pi(C_n).
\]  
We call these cylinders, \textit{cylinders of level $n$} or \textit{$n$-level cylinders}.

\begin{lemma} \label{Lemma cylinders in shift space and in manifold are of comparable diameter}
There exists a $c>1$ such that, for all $n\in \mathbb{N}$ and all cylinders $C_n \subset \Sigma$ and $\widetilde{C}_n \subset \mathbb{T}^2$
\[
c^{-1} |\widetilde{C}_n| \ \leq \ |C_n| \ \leq \ c |\widetilde{C}_n|.
\] 
\end{lemma}
\begin{proof}
Firstly we have that 
\[
|C_n| = \frac{1}{\lambda^{n+1}}.
\]
Now by the definition of the cylinder in $\mathbb{T}^2$ we have that 
\[ c_{\min} \ \lambda^{-n} \leq |\widetilde{C}_n| \leq c_{\max} \ \lambda^{-n}
\]
where $c_{\min}:= \min\{ |X_j|: \ X_j \text{ is an element of the Markov partition} \}$ and $c_{\max}$ is accordingly defined. Since this is a finite partition consisting of compact elements $0 < c_{\min}, c_{\max} < +\infty$.
\end{proof}

\begin{lemma} \label{Lemma comparability of balls and cylinders in 2D torus}
Let $\widetilde{C}_n$ be an $n$-level cylinder in $\mathbb{T}^2$. Then there exists a ball $B_n$ such that 
\[
\widetilde{C}_n \subset B_n \quad \text{ and } \quad |B_n|\leq 3 |\widetilde{C}_n| .
\]
On the other hand, there exist absolute constants $c'>0$ and $k_0 \in \mathbb{N}$ such that if $B(x, r)$ is a ball in $\mathbb{T}^2$, then one needs at most $k_0$ cylinders of level $n=n(r)$ so that
\[
B(x,r) \subset \bigcup_{i=1}^{k_{0}} \widetilde{C}_{n}^i \quad \text{ and } \quad | \widetilde{C}_{n}^i| \leq c' |B(x,r)|.
\] 
\end{lemma}
\begin{proof}
The first part is clear. All we have to do is consider a ball with centre one point in the cylinder and radius equal to the diameter of the cylinder.
Now for the second part, there exists an $n_0\in \mathbb{N}$ such that $\lambda^{-n_0-1} \leq r < \lambda^{-n_0}$. By Lemma~\ref{Lemma cylinders in shift space and in manifold are of comparable diameter}, we have that the $n_0$-level cylinders have diameter comparable to $\lambda^{-n_0}$, and in particular it is at most $c_{\max} \lambda^{-n_0}$ (see the proof of the same lemma). Also the diameter of the ball is $2r<2 \lambda^{-n_0}$. Therefore, for every $n_0$-level cylinder,
\[
|C_{n_0}| \ \leq \ 2 c_{\max} |B(x,r)|.
\]

Let us consider $b_{\min}$ to be the minimum length of all sides of all elements of the partition and $h_{\min}$ to be the minimum of all heights of all elements of the partition. 
\begin{figure} [H]
\begin{tikzpicture}
[declare function={a=3;b=1.2;alpha=65;}]    
\path 
(0,0) coordinate (A) +(alpha/2:.6) 
(a,0) coordinate (B)
(alpha:b) coordinate (D)
($(B)+(D)-(A)$) coordinate (C)
($(A)!(D)!(B)$) coordinate (H)
;
\draw[dashed] (D)--(H) node[midway,right]{$h$};
\draw[thick] (A)
--(B) node[midway,below]{$b$}
--(C)--(D)
--cycle node[midway,left]{};
;
\end{tikzpicture}
\end{figure}
Then the side and the height of an $n$-level cylinder is at least $\lambda^{-n}b_{\min}$ and $\lambda^{-n}h_{\min}$ respectively. Since no overlapping of the (interiors) occurs between cylinders of the same level, by considering the diameter that is parallel to the $h_{\min}$, we can see that the number of $n_0$-level cylinders needed to cover a distance equal to the diameter of the ball is at most
\[
\left\lceil \frac{2r}{h_{\min}\lambda^{-n_0}} \right\rceil + 1.
\]
Again, since there is no overlapping of (the interiors of) the cylinders of a given level and they cover the whole space and since $h_{\min} \leq b_{\min}$, the number of $n_0$-level cylinders covering the ball is at most 
\[
\biggl( \left\lceil \frac{2r}{h_{\min}\lambda^{-n_0}} \right\rceil + 1 \biggr)^2
\]
which is smaller than
\[
\biggl( \left\lceil\frac{2\lambda^{-n_0}}{h_{\min}\lambda^{-n_0}}
\right\rceil + 1 \biggr)^2 \ 
= \ 
\biggl( \left\lceil\frac{2}{h_{\min}} \right\rceil + 1 \biggr)^2.
\]
Now for $c'=2 c_{\max}$ and $k_0= \Bigl(
\left\lceil\frac{2}{h_{\min}} \right\rceil + 1 \Bigr)^2$ we have the assertion.
\end{proof}

A straightforward corollary of Lemma~\ref{Lemma comparability of balls and cylinders in 2D torus} is the following.

\begin{corollary} \label{Cor. cylinders are sufficient for the
    dimension in 2D torus} In $\mathbb{T}^2$ one needs to
  consider only covers of cylinders in order to calculate the
  Hausdorff dimension of any subset of\/ $\mathbb{T}^2$.
\end{corollary}

The combination of Lemma~\ref{Lemma cylinders in shift space and in manifold are of comparable diameter} and Corollary \ref{Cor. cylinders are sufficient for the dimension in 2D torus} gives Proposition \ref{Proposition dim of E in shift space  equals  dim of p(E) in manifold through coding map}.

As it was established in Proposition \ref{Proposition dim of E in shift space  equals  dim of p(E) in manifold through coding map}, it is sufficient to work with the set 
\[
\widehat{\mathcal{U}}(\alpha) \ 
:= \ 
\pi^{-1} \big( \mathcal{U}(\alpha)  \big) 
\ \subset \ 
\Sigma_{\Gamma}.
\]
In other words, we will show in reality that 
\[
\dim_{H}\big(\widehat{\mathcal{U}}(\alpha) \big)
=
\dim_{H} \big(\mathcal{U}'(\alpha) \big).
\]
We will proceed with a series of lemmas that, in the end, will provide the assertion.

\begin{lemma} \label{Lemma coding map is Lipschitz}
The coding map $\pi \colon \Sigma_{\Gamma} \to \mathbb{T}^{2}$, with the chosen metrics,
is $L$-Lipschitz for some $L>0$.  
\end{lemma}
\begin{proof}
Let $\delta>0$ and $\underline{x}, \underline{y} \in \Sigma_{\Gamma}$. Assume that
\[
d_{\Sigma}(\underline{x}, \underline{y}) < \delta.
\]
There exists an $m>0$ such that 
\[
\lambda^{-m-1} \ < \ \delta \ \leq \ \lambda^{-m}
\]
which implies that $\underline{x}, \underline{y} \in C_{m}$. 
Therefore, $x,y \in \widetilde{C}_m$, where $x= \pi(\underline{x})$ and $y= \pi(\underline{y})$. Hence 
\[
d_{\mathbb{T}^2} (x,y) \leq c_{\max} \lambda^{-m}
\]
where, as before, $c_{\max}:= \max\{ |X_j|: \ X_j \text{ is an element of the Markov partition} \}$. 
For $L=c_{\max}$ we have the assertion.
\end{proof}

Observe that by changing in our covering arguments the length of the fixed blocks in the proof of~\eqref{Equation I: lower bound} or in the proof of~\eqref{Equation II: lower bound} by a constant quantity, it does not affect the dimension whatsoever. Therefore we get the following result essentially for free.
\begin{lemma} \label{Lemma constant in ell_N doesn't change the dimension}
Let $c>0$ be a constant and consider the set
\[
\mathcal{U}_{c}'(\alpha) = \{\underline{x}\in \Sigma_{\Gamma}: \ \exists M=M(\underline{x}) \geq 1 \text{ such that } \forall N\geq M, \ \exists n\leq N \text{ such that } d_{\Sigma}(\sigma^n\underline{x}, \underline{x}) \leq c \cdot\lambda^{-\alpha N}\}.
 \]
Then
\[
\dim_{H} \big( \mathcal{U}'(\alpha) \big) = \dim_{H} \big( \mathcal{U}_{c}'(\alpha) \big).
\]
\end{lemma}

\begin{corollary} \label{Cor. U hat  >  U'}
\[
\dim_{H}\big(\widehat{\mathcal{U}}(\alpha) \big)
\geq
\dim_{H} \big(\mathcal{U}'(\alpha) \big).
\]
\end{corollary}
\begin{proof}
Since, $\pi$ is $L$-Lipschitz (Lemma~\ref{Lemma coding map is Lipschitz}), we have 
\[
\widehat{\mathcal{U}}(\alpha) 
\supset
\mathcal{U}_{L^{-1}}'(\alpha).
\]
By Lemma~\ref{Lemma constant in ell_N doesn't change the dimension} we get the result.
\end{proof}

Now the more interesting part is to show the other inequality. By looking a little bit more carefully at the structure of the set $\widehat{\mathcal{U}}(\alpha)$ and the clear commonalities with its counterpart $\mathcal{U}'(\alpha)$, one can show that a cover can be constructed for the $\widehat{\mathcal{U}}(\alpha)$, following the method of Bugeaud and Liao. The only difference is that for each "fixed block", we now have $K_0$ many blocks that we could choose from \big(instead of just one, as for the cover of $\mathcal{U}'(\alpha)$\big), where $K_0>0$ is an absolute constant but this does not affect the dimension. We omit a big portion of the details for the construction of the cover, as they have been already clarified in the proof of~\eqref{Equation II: lower bound}.
\begin{lemma} \label{Lemma dim (Uhat)  <  dimU'}
\[
\dim_{H}\big( \widehat{\mathcal{U}}(\alpha) \big) \ \leq \ \dim_{H} \big( \mathcal{U}'(\alpha) \big).
\]
\end{lemma}
\begin{remark}
We deliberately omit the Lipschitz constant $L$---basically treating it as if it could be chosen to be $1$---for convenience, in the proof of Lemma~\ref{Lemma dim (Uhat)  <  dimU'}. This does not alter the results nor the ideas and methods of the proof, while making the computations slightly simpler and the ideas more clear. Besides, we can actually choose $L$ to be less or equal to $1$ since we can choose a partition of arbitrarily small diameter, see for example \cite[Chapter 18]{Katok Hasselblatt}. Since $L:=c_{\max} :=\max\{ \text{diameters of all elements of the partition} \}$ we have that $L$ can indeed be chosen to be less or equal to $1$. Of course we have by Proposition \ref{Poposition entropy of SFT} and Corollary \ref{Corollary admissible blocks are almost lambda to the n} that, the largest eigenvalue of the corresponding transition matrix is still $\lambda$.
\end{remark}
\begin{proof}
Consider a point $\underline{x} \in \widehat{\mathcal{U}}(\alpha)$. This implies that for all large $N$, there exists an $n\in \{1, \ldots , N \}$ so that,
\begin{equation} \tag{a} \label{eq.  < lamba_N}
d_{\mathbb{T}^2} \big( T^n (\pi(\underline{x})) , \pi(\underline{x}) \big)
<
\lambda^{-\alpha N}.
\end{equation} 

Let $\xi:= \lfloor \alpha N \rfloor +  1 \in \mathbb{N}$ and consider the the $\xi$-th level cylinder $\widetilde{C}_{\xi}$ which contains the point $x:=\pi(\underline{x})$. By the relation \eqref{eq.  < lamba_N}, either $A^nx$ is contained in $\widetilde{C}_{\xi}$, or in a $\lambda^{-\alpha N}$-region of $\widetilde{C}_{\xi}$, which we denote as $\Delta_{\xi}$. A similar argument as in the proof of Lemma~\ref{Lemma comparability of balls and cylinders in 2D torus} gives us that there exists a an absolute constant (i.e.\ independent from $N$) $K_0\in \{1,2, \ldots \}$ such that, at most $K_0$ cylinders of level $\xi$ are needed in order to cover $\Delta_{\xi}$, i.e.\ to describe the relation \eqref{eq.  < lamba_N}.

In the same way as in the proof of~\eqref{Equation II: lower bound}, we construct a a cover for $\widehat{\mathcal{U}}(\alpha)$, but now we have to take also into account the $K_0$ choices that we have for each of the "fixed" blocks. More precisely, for the appropriate $k$'s so that $k\sim \log n$, where $n$ corresponds to the centre of the $k$-th fixed block, we have $K_0^{k} \sim K_0^{\log n}$ many choices if we have $k$ many "fixed" blocks.
Therefore we need for the covering argument, in analogy of \eqref{eq:upper:number} and \eqref{eq:upper:number2} in the proof of~\eqref{Equation II: lower bound},
\[
C_2 \log n \cdot n^{C_2 \log n} \cdot \lambda^{P(n)} \cdot K_0^{\log n}
\]
many cylinders in the first case and 
\[
C_2 \log n \cdot n^{C_2 \log n} \cdot \lambda^{Q(n)}\cdot K_0^{\log n}
\]
many cylinders in the second case.
This does not alter the critical value for neither of the two cases and therefore we get the same upper bounds for $\dim_{H}\big( \widehat{\mathcal{U}}(\alpha) \big)$ as we did for $ \dim_{H} \big( \mathcal{U}'(\alpha) \big)$. In particular
\[
\dim_{H}\big( \widehat{\mathcal{U}}(\alpha) \big) \ \leq \ \dim_{H} \big( \mathcal{U}'(\alpha) \big).
\qedhere
\]
\end{proof}

\end{document}